\documentclass[12pt,reqno]{amsart}

\usepackage{hyperref}

\usepackage[T1]{fontenc}
\usepackage[latin1]{inputenc}
\usepackage{latexsym}
\usepackage[arrow, matrix, curve]{xy}
\usepackage[dvips]{graphicx}
\usepackage{epsfig}
\usepackage{bm}
\usepackage{accents}
\usepackage{amssymb}
\usepackage{amsmath}
\usepackage{verbatim}
\usepackage{amsthm}
\usepackage{enumerate}
\usepackage{mathrsfs}
\usepackage[hmargin=3cm,vmargin=3cm]{geometry}
\newtheorem{theorem}{Theorem}[section]
\newtheorem{lemma}[theorem]{Lemma}
\newtheorem{corollary}[theorem]{Corollary}

\theoremstyle{definition}

\theoremstyle{remark}

\numberwithin{equation}{section}

\def\ba{{\mathbf a}}
\def\bb{{\mathbf b}}
\def\bc{{\mathbf c}}
\def\bd{{\mathbf d}}
\def\be{{\mathbf e}}

\def\bh{{\mathbf h}}

\def\bj{{\mathbf j}}
\def\bk{{\mathbf k}}

\def\bm{{\mathbf m}}

\def\br{{\mathbf r}}

\def\bt{{\mathbf t}}
\def\bu{{\mathbf u}}
\def\bv{{\mathbf v}}
\def\bw{{\mathbf w}}
\def\bx{{\mathbf x}}
\def\by{{\mathbf y}}
\def\bz{{\mathbf z}}

\def\calB{{\mathcal B}} 
\def\calC{{\mathcal C}}

\def\calF{{\mathcal F}}
\def\calG{{\mathcal G}}

\def\calI{{\mathcal I}}
\def\calJ{{\mathcal J}}

\def\calN{{\mathcal N}}

\def\calR{{\mathcal R}}

\def\scrJ{{\mathscr J}}

\def\scrS{{\mathscr S}}

\def\C{\mathbb C}

\def\N{\mathbb N}

\def\Q{\mathbb Q}
\def\R{\mathbb R}

\def\Z{\mathbb Z}

\def\frB{{\mathfrak B}}

\def\frG{{\mathfrak G}}
\def\frH{{\mathfrak H}}

\def\frJ{{\mathfrak J}}
\def\frk{{\mathfrak k}}

\def\frm{{\mathfrak m}}\def\frM{{\mathfrak M}}

\def\frS{{\mathfrak S}}

\def\alp{{\alpha}}  
\def\bet{{\beta}}
\def\gam{{\gamma}} \def\Gam{{\Gamma}} 
\def\del{{\delta}} \def\Del{{\Delta}}
\def\eps{\varepsilon}

\def\tet{{\vartheta}}

 \def\Lam{{\Lambda}}

\def\sig{{\sigma}} 
\def\Ups{{\Upsilon}} 
\def\ome{{\omega}} \def\Ome{{\Omega}}

\def\balp{{\boldsymbol \alpha}}
\def\bbet{{\boldsymbol \beta}}
\def\bgam{\boldsymbol \gamma} 

\def\bdel{\boldsymbol \delta} \def\bDel{\boldsymbol \Delta}

\def\bfeta{{\boldsymbol \eta}}
\def\btet{{\boldsymbol \vartheta}}

 \def\bLam{{\boldsymbol \Lambda}}

\def\bxi{{\boldsymbol \xi}}

\def\bsig{{\boldsymbol \sig}}

\def\d{{\partial}}

\def\l{\ell}
\def\le{\leqslant} \def\ge{\geqslant}

\def\d{{\,{\rm d}}}

\def\msum#1#2{\sum_{\substack{ {1 \le {#2} \le {#1}} \\ {({#1}, {#2})=1} }}}

\DeclareMathOperator{\card}{card}

\DeclareMathOperator{\vol}{vol}

\usepackage{hyperref}

\begin{document}
\title[Optimal mean value estimates]{Optimal mean value estimates beyond\\ Vinogradov's 
mean value theorem}
\author[Julia Brandes]{Julia Brandes}
\address{Mathematical Sciences, University of Gothenburg and Chalmers Institute of 
Technology, 412 96 G\"oteborg, Sweden}
\email{brjulia@chalmers.se}
\author[Trevor D. Wooley]{Trevor D. Wooley}
\address{Department of Mathematics, Purdue University, 150 N. University Street, West 
Lafayette, IN 47907-2067, USA}
\email{twooley@purdue.edu}
\subjclass[2010]{11L15, 11D45, 11L07, 11P55}
\keywords{Exponential sums, Hardy--Littlewood method}
\date{}
\begin{abstract} We establish improved mean value estimates associated with the number 
of integer solutions of certain systems of diagonal equations, in some instances attaining 
the sharpest conjectured conclusions. This is the first occasion on which bounds of this 
quality have been attained for Diophantine systems not of Vinogradov type. As a 
consequence of this progress, whenever $u \ge 3v$ we obtain the Hasse principle for 
systems consisting of $v$ cubic and $u$ quadratic diagonal equations in $6v+4u+1$ 
variables, thus attaining the convexity barrier for this problem.
\end{abstract}
\maketitle

\section{Introduction}\label{S1}
In recent years, our understanding of systems of diagonal equations and their associated 
mean values has advanced rapidly. Whilst only a few years ago, such mean values had 
been comprehensively understood only in the most basic cases, the resolution of the main 
conjecture associated with Vinogradov's mean value theorem by the second author 
\cite{Woo2016,Woo2017} and Bourgain, Demeter and Guth \cite{BDG2016} has 
transformed the landscape. It now seems feasible to address the challenge of establishing 
similarly strong results for a much wider class of cognate problems.\par

In this memoir, we make progress towards, and in certain cases attain, the convexity 
barrier for a family of mean values associated with systems of equations that fail to be 
translation-dilation invariant and thus lie outside the scope of the efficient congruencing 
and $\l^2$-decoupling methods developed by the second author 
\cite{Woo2016, Woo2017} and Bourgain, Demeter and Guth \cite{BDG2016}. The most 
accessible of our results addresses systems of cubic and quadratic diagonal equations. Let 
$\calN_{s,v,u}(X)$ denote the number of integral solutions $\bx \in [-X,X]^s$ of the 
system of equations
\begin{equation}\label{1.1}
    \begin{aligned}
   		c_{i,1}^{(3)} x_1^3 + \ldots + c_{i,s}^{(3)} x_s^3 &= 0 \qquad (1 \le i \le v) \\
   		c_{j,1}^{(2)} x_1^2 + \ldots + c_{j,s}^{(2)} x_s^2 &= 0 \qquad (1 \le j \le u),
    \end{aligned}
\end{equation}
consisting of $u$ quadratic and $v$ cubic equations of diagonal shape. Here and 
throughout we assume the coefficients $c_{i,j}^{(k)}$ of such systems to be integral. It 
is clear that the presence of coefficients in such systems necessitates some kind of 
non-singularity condition, lest the equations interact in some non-generic way. We refer to 
an $r \times s$ matrix $C$ as \emph{highly non-singular} if $s \ge r$ and any collection 
of $r$ distinct columns of $C$ forms a non-singular matrix.\par

Our first result shows that $\calN_{s,v,u}(X)$ satisfies the anticipated asymptotic formula 
for all sets of coefficients in general position, provided that $s \ge 6v+4u+1$ and 
$u \ge 3v$. This achieves the conjectured convexity barrier.
\begin{theorem}\label{T1.1}
    Suppose that $u \ge 3v$ and that $s \ge 6v + 4u + 1$. Assume further that the 
coefficient matrices 
    \begin{align*}
    	C^{(2)} = (c_{i,j}^{(2)})_{\substack{1 \le i \le u \\ 1 \le j \le s}} \quad \text{and} 
\quad C^{(3)} = (c_{i,j}^{(3)})_{\substack{1 \le i \le v \\ 1 \le j \le s}}
    \end{align*} 
    are highly non-singular. Then there exist constants $\calC \ge 0$ and $\delta>0$ such 
that
    \begin{align}\label{1.2}
        \calN_{s,v,u}(X) = \calC X^{s-3v-2u} + O(X^{s-3v-2u-\delta}).
    \end{align}
    Moreover, if the system \eqref{1.1} has non-singular real and $p$-adic solutions for all 
primes $p$, then $\calC > 0$. 
\end{theorem}
In general, asymptotic formul{\ae} like the one supplied by \eqref{1.2} are expected to 
hold whenever the number of variables exceeds twice the total degree of the system. 
However, thus far the validity of such an asymptotic formula has been proved only in a 
few isolated instances. Arguably the first non-trivial case in which this convexity barrier 
was achieved occurs in work of Cook \cite{Coo1971, Coo1973} concerning pairs and 
triples of diagonal quadratic equations. Recent work of Br\"udern and the second author 
\cite{BW2007, BW2016} obtains asymptotic lower bounds at the convexity limit for 
systems of diagonal cubic forms. In the case of mixed systems of cubic and quadratic 
equations, work of the second author underlying \cite[Theorem~1.2]{Woo2015} achieves 
the convexity limit in the case $u=v=1$ with $s\ge 11$ relating to systems consisting of 
one cubic and one quadratic diagonal equation. Most recently, investigations of the first 
author joint with Parsell \cite[Theorem~1.4]{BP2017} establish an asymptotic formula 
tantamount to \eqref{1.2} for systems of $v$ cubic and $u$ quadratic diagonal equations, 
though under the more restrictive hypothesis that $s \ge \lfloor 20v/3 \rfloor + 4u + 1$, 
thus missing the convexity barrier whenever $v \ge 2$. In subsequent work \cite{B2017}, 
the first author proved that an asymptotic formula of the shape \eqref{1.2} holds when 
$v \ge 2u$ and $s \ge 6v + \lfloor 14u/3 \rfloor + 1$, which misses the convexity barrier 
when $u \ge 2$. Thus, Theorem~\ref{T1.1} provides the first instance where bounds of the expected quality have been achieved for systems of $v$ cubic and $u$ quadratic equations in settings where both $u$ and $v$ exceed $1$.\par

Theorem~\ref{T1.1} is in fact a special case of our more general Theorem~\ref{T1.5} below. Both of these results rest on our new estimates for certain mean values of Vinogradov type. In their most general form, such mean values encode the number of integral solutions of systems of the general shape
\begin{align}\label{1.3}
    c_{i,1}^{(l)} (x_1^l-y_1^l) + \ldots + c_{i,s}^{(l)} (x_s^l - y_s^l) = 0 \qquad (1 \le i \le r_l,\, 1 \le l \le k),
\end{align}
in which $r_1, \dots, r_k$ are non-negative integers and the coefficients $c_{i,j}^{(l)}$ 
are integers. When all of the coefficient matrices 
\[
	C^{(l)} = \bigl(c_{i,j}^{(l)}\bigr)_{\substack{ 1 \le i \le r_l\\1 \le j \le s}}
\]
are highly non-singular, then the main conjecture states that the number of integral solutions $\bx, \by \in [-X,X]^s$ of the system \eqref{1.3} should be at most of order $X^{s+\eps} + X^{2s - K}$, for any $\eps>0$, where $K=r_1+2r_2+\ldots+kr_k$ denotes the system's total degree. A corresponding lower bound, with $\eps=0$, is provided by an argument akin to that delivering \cite[equation (7.4)]{Vau}. Systems of the shape \eqref{1.3} have previously been studied by the first author together with Parsell \cite{BP2017}, where it was shown that the main conjecture for such systems holds when $r_l \ge r_{l+1}$ for all $1 \le l \le k-1$. In the latter circumstances, the system \eqref{1.3} can be viewed as a superposition of Vinogradov systems of various degrees (see Theorem~2.1 and Corollary~2.2 in that paper). In wider generality, bounds of the strength of those described in Theorems \ref{T1.1} and \ref{T1.5} were known hitherto only for systems of quadratic equations and systems of Vinogradov type, as well as superpositions of these two special classes of systems.\par

The goal of the work at hand is to enlarge the range of systems of type \eqref{1.3} for which the main conjecture is known to hold. When the coefficient matrices $C^{(l)}$ are highly non-singular for $2 \le l \le k$, we denote by 
$I_{s,k}^{v,u}(X) = I_{s,k}^{v,u}(X; C^{(2)}, \ldots, C^{(k)})$ the number of integral 
solutions $\bx, \by \in [-X,X]^s$ of the system \eqref{1.3}, where 
\[
	r_k = v , \qquad r_{k-1} = \ldots = r_2 = u, \qquad r_1=0. 
\] 
Write further 
\begin{align}\label{1.4}
		K={\textstyle\frac12} k(k-1)u + kv - u,
\end{align}	
so that $K$ denotes the total degree of the system. 

In order to describe our new results concerning the mean value $I_{s,k}^{v,u}(X)$, we need to consider certain auxiliary systems of equations. Let $l \ge 2$  be an integer and write $\sig = \frac12 l(l+1)$. Then, given a positive number $X$, we denote by $M^*_l(X)$ the number of integer tuples $\bx, \by \in [-X,X]^{\sig-2}$ and $\bz,\bh \in [-X,X]^2$ satisfying 
\begin{align}\label{1.5}
    \sum_{i=1}^{\sig-2}(x_i^j-y_i^j)+j(z_1^{j-1}h_1+z_2^{j-1}h_2)=0\qquad (1\le j\le l).
\end{align} 
The main conjecture for systems of the shape \eqref{1.5} claims that 
\begin{align}\label{1.6}
	M^*_l(X) \ll X^{\frac12 l(l+1) + \eps}.
\end{align}
Our first main result is as follows. 

\begin{theorem}\label{T1.2}
	Suppose that $k \ge 3$, $v \ge 1$ and $u \ge 2v$ are integers with $u|kv$, and assume \eqref{1.6} for $l=k-1$. Then for any $s \ge u$ and any $\eps > 0$ we have
	\begin{align*}
		I_{s,k}^{v,u}(X) \ll X^{\eps}(X^s + X^{2s-K}).
	\end{align*}
\end{theorem}

By combining the ideas of the proof of Theorem~\ref{T1.2} with those underlying \cite[Theorem~2.1]{BP2017}, we can extend our results to cover also superpositions of systems of equations of the kind considered in Theorem~\ref{T1.2}. Fix a collection of degrees $k_1>k_2>\ldots >k_n\ge 3$ with associated multiplicities $v_1,\ldots,v_n\in \N$. Moreover, fix a tuple $u_0, u_1, \ldots, u_n$ of non-negative integers with $u_i \ge v_i$ for $1 \le i \le n$, set $k=k_1$, and define $w_0=0$ and $w_i=u_1+\ldots+u_i$ for $1\le i\le n$. Now define the parameter $r_l$ by putting
\begin{align}\label{1.7}
	r_1=0, \qquad r_2=w_n+u_0, \qquad r_l = \begin{cases}
	w_n & \text{when $3 \le l < k_n$}, \\
	w_j & \text{when $k_{j+1} < l < k_j$}, \\
	w_{j-1}+v_j& \text{when $l = k_j$ for some $j$}.
 		\end{cases}
\end{align}
We denote by 
\begin{equation}\label{1.8}
	I_{s,\bk}^{\bv,\bu}(X)=I_{s,\bk}^{\bv,\bu}(X; C^{(2)}, \ldots, C^{(k)})
\end{equation}
the number of integer solutions $\bx, \by \in [-X,X]^s$ of the system \eqref{1.3} with $\br$ defined as in \eqref{1.7}. These systems can be viewed as superpositions of systems of the shape considered in Theorem~\ref{T1.2} with parameters $(k_j,v_j,u_j)$, together with $u_0$ additional quadratic equations. Here, the total degree is given by 
\begin{align}\label{1.9} 
	K = \sum_{l=2}^k l r_l = \sum_{j=1}^n K_j +2u_0,
\end{align}
where, in accordance with \eqref{1.4}, we write
\[
	K_j=\tfrac{1}{2}k_j(k_j-1)u_j+k_jv_j-u_j.
\]
In this notation, we have the following generalisation of Theorem~\ref{T1.2}. 

\begin{theorem}\label{T1.3}
	Let $u_0$ be a non-negative integer. Suppose that $k_1 > \ldots > k_n \ge 3$, and let 
$u_1, \ldots, u_n$ as well as $v_1, \ldots, v_n$ be natural numbers with $u_j \ge 2 v_j$ 
and $u_j | k_j v_j$ for $1 \le j \le n$. Also, assume \eqref{1.6} for all degrees $l=k_j-1$ 
with $1 \le j \le n$. Then for $s \ge r_2$ and any $\eps>0$ we have 
	\begin{align*}
		I_{s,\bk}^{\bv,\bu}(X) \ll X^{\eps}(X^s + X^{2s-K}).
	\end{align*}
\end{theorem}
We also have an alternative, unconditional formulation of this result, which is given in Theorem~\ref{T3.3} below.

\par To illustrate the strength of our results in Theorems \ref{T1.2} and \ref{T1.3}, we discuss in more detail some of the most relevant special cases. Among the systems of diagonal equations not of Vinogradov type, the most well-studied ones are systems of cubic equations and systems of cubic and quadratic equations, such as we considered in our motivating example in Theorem~\ref{T1.1}. Regarding such systems, it is immediate from work of the second author \cite[Theorem~1.1]{Woo2015} that for every $\eps >0$ one has $I_{5,3}^{1,1}(X) \ll X^{31/6+\eps}$, and this bound implies via \cite[Theorem~2.1]{BP2017} that $I_{3+2u,3}^{1,u}(X) \ll X^{3+2u+1/6+\eps}$ for all $u \ge 1$. Theorem~\ref{T1.3} now allows us to improve this result.

\begin{corollary}\label{C1.4}
	Suppose that $v \ge 1$ and $s \ge u \ge 3v$. Then for any $\eps>0$ we have
	\begin{align*}
	I_{s,3}^{v,u}(X) \ll X^\eps (X^s + X^{2s-3v-2u}).
	\end{align*}
\end{corollary}
This follows from Theorem~\ref{T1.3} in combination with Lemma~\ref{L2.1} below. 
Corollary~\ref{C1.4} represents only the second occasion, after the second author's 
successful treatment of the cubic case of Vinogradov's mean value theorem 
\cite{Woo2016}, that the convexity barrier has been attained for a system of diagonal 
equations involving cubic equations. In particular, we now have the main conjecture for 
mean values that correspond to systems consisting of one cubic and three quadratic 
diagonal equations.~This is the main new input that enables us to prove 
Theorem~\ref{T1.1}.\par

Our results complement older ones that can be obtained by other means. On the one hand, it follows from Theorem~2.1 and Corollary~2.2 of the first author's work with Parsell \cite{BP2017} in combination with Vinogradov's mean value theorem \cite[Theorem~1.1]{BDG2016} that the conclusion of Theorem~\ref{T1.3} holds unconditionally in the range 
\begin{align*}
	s \ge \sum_{j=1}^n \left( v_j \frac{k_j (k_j+1)}{2} + (u_j-v_j) 
\frac{k_j (k_j-1)}{2}\right) + 2 u_0 = K + w_n.
\end{align*}
On the other hand, for small $s$ the second author's result  
\cite[Corollary~1.2]{Woo2017} can be combined with the arguments of 
\cite[Theorem~2.1]{BP2017} to establish the conclusion of Theorem~\ref{T1.3} 
unconditionally in the range 
\begin{align*}
	s &\le \sum_{j=1}^n \left(v_j \frac{k_j(k_j-1)}{2} + (u_j-v_j) 
\frac{(k_j-1)(k_j-2)}{2} \right)+2u_0\\
	& = K - \sum_{j=1}^n ((k_j-2)u_j + v_j ). 
\end{align*}

\par Mean value estimates like those of Theorems \ref{T1.2} and \ref{T1.3} have long been employed to establish asymptotic formul{\ae} for the number of solutions of simultaneous diagonal equations. For $\br$ as in \eqref{1.7} and highly non-singular coefficient matrices $C^{(l)}$ $(2 \le l \le k)$, denote by $N_{s, \bk}^{\bv, \bu}(X)$ the number of integral solutions of the system of equations
\begin{align}\label{1.10}
    c_{i, 1}^{(l)}x_1^l + \ldots + c_{i, s}^{(l)}x_s^l=0 \qquad (1\le i\le r_l, \, 2\le l\le k)
\end{align}
with $|x_j| \le X$ for $1 \le j\le s$. It is well known that, if $s$ is sufficiently large in terms of $\bk$, $\bv$ and $\bu$, there is an asymptotic formula of the shape
\begin{align}\label{1.11}
    N_{s, \bk}^{\bv, \bu}(X) = (\calC + o(1))X^{s-K},
\end{align}
where $\calC$ is a non-negative constant encoding the local solubility data for the system \eqref{1.10}. The relevant question is how large $s$ has to be for an asymptotic formula 
like that of \eqref{1.11} to hold. Theorem~1.1 of \cite{BP2017} provides a bound for $s$ 
that is somewhat unwieldy, but can likely be reduced to
\begin{align*}
    s \ge \sum_{j=1}^n \big(v_jk_j(k_j+1) +(u_j-v_j)k_j(k_j-1)\big) + 4u_0 +1= 
2K + 2w_n + 1
\end{align*}
by accounting for our revised treatment of the major arcs described in \S\S\ref{S5}--\ref{S6} below. On the other hand, unless fundamentally new methods become available that avoid the use of mean values, we cannot expect to be able to establish such asymptotic formul{\ae} when $s \le 2K$. Thanks to our new mean value estimates in Theorem~\ref{T1.3}, we are now able to make progress towards this theoretical barrier.

\begin{theorem}\label{T1.5}
	Let $u_0$ be a non-negative integer. Suppose that $k_1 > \ldots > k_n \ge 3$, and let 
$u_1, \ldots, u_n$ as well as $v_1, \ldots, v_n$ be natural numbers with $u_j \ge 2 v_j$ 
and $u_j | k_jv_j$ for $1 \le j \le n$. Also, assume \eqref{1.6} for all degrees $l=k_j-1$ 
with
 $1 \le j \le n$. Then for $s \ge 2K+1$ the asymptotic formula \eqref{1.11} holds with 
$\calC \ge 0$. If, furthermore, the system \eqref{1.10} has non-singular solutions in $\R$ 
as well as in the fields $\Q_p$ for all $p$, then the constant $\calC$ is positive.
\end{theorem}

Again, we refer to Theorem~\ref{T4.1} below for an unconditional version of this result. 
Moreover, we note that in Lemma~\ref{L2.1} below it is shown that the bound 
\eqref{1.6} holds for $l=2$, and thus Theorem~\ref{T1.1} can be deduced as a special 
case of Theorem~\ref{T1.5}, corresponding to the parameters $k = 3$ and $u_0 = 
u-3v$.\par

The proofs of our results rest on an idea that played a crucial role in the second author's 
work on pairs of quadratic and cubic diagonal equations \cite{Woo2015}, and which has 
been explored further in the authors' recent work on incomplete Vinogradov systems 
\cite{BW2017}. In these papers, the missing linear equation is artificially added in, which 
makes it possible to exploit the strong bounds on Vinogradov's mean value theorem. By 
taking advantage of the translation-dilation invariance of the newly completed Vinogradov 
systems, we then relate these systems to the auxiliary mean values $M^*_{l}(X)$ 
introduced above. Whilst our understanding of these auxiliary mean values remains 
unsatisfactory for general degree, the quantity $M^*_{2}(X)$ may be comprehensively 
understood in terms of quadratic Vinogradov systems. This observation plays a pivotal role 
in our argument, and it is the main reason why we attain the convexity barrier in Theorem 
\ref{T1.1} and Corollary \ref{C1.4}.\par

\noindent {\textbf{Notation.}}
Throughout, the letters $s$, $u$, $v$, and $k$, as well as the entries of the vectors 
$\bk$, $\bu$, $\bv$, and $\br$, will denote non-negative integers. The letter $\eps$ will 
be used to denote an arbitrary, but sufficiently small  positive number, and we adopt the 
convention that whenever it appears in a statement, we assert that the statement holds 
for all sufficiently small $\eps>0$. We take $X$ to be a large positive number which, just 
like the implicit constants in the notations of Landau and Vinogradov, is permitted to 
depend at most on $s$, $\bk$, $\bv$, $\bu$, the coefficient matrices $C^{(l)}$ 
$(2 \le l \le k)$, and $\eps$. We employ the non-standard notation that when 
$G:[0,1)^n \rightarrow \C$ is integrable for some $n \in \N$, then
\begin{align*}
    \oint G(\balp) \d \balp = \int_{[0,1)^n}G(\balp) \d \balp.
\end{align*}
Here and elsewhere, we use vector notation liberally in a manner that is easily discerned 
from the context. In particular, when $\bb$ denotes the integer tuple $(b_1, \dots, b_n)$, 
we write $(q, \bb) = \gcd(q, b_1, \dots, b_n)$.\par

\noindent {\textbf{Acknowledgements.}}
Both authors thank the Fields Institute in Toronto for excellent working conditions and 
support that made this work possible during the Thematic Program on Unlikely 
Intersections, Heights, and Efficient Congruencing. This work was further facilitated by 
subsequent visits of the first author to the University of Bristol, and of the second author 
to the University of Waterloo. The authors gratefully acknowledge the hospitality of both 
institutions. The work of both authors was supported by the National Science Foundation 
under Grant No.~DMS-1440140 while they were in residence at the Mathematical Sciences 
Research Institute in Berkeley, California, during the Spring 2017 semester. The first 
author's work was supported in part by Starting Grant 2017-05110 from 
Vetenskapsr{\aa}det. The second author's work was supported by a European Research 
Council Advanced Grant under the European Union's Horizon 2020 research and innovation 
programme via grant agreement No.~695223, and in the final stages by the National 
Science Foundation via Grant No.~DMS-1854398 and DMS-2001549.\par

The authors are also very grateful to Scott Parsell for pointing out an oversight in an 
earlier version of this paper, which necessitated a fundamental re-write of parts of the 
argument.

\section{Preliminaries and preparatory steps}\label{S2}
Our goal in this and the next section is the proof of Theorem~\ref{T1.3}. Before delving 
to the core of the argument, we pause to introduce some notation and establish a mean 
value estimate that will be of use in our subsequent discussion. For $2 \le l \le k$ we 
define the exponential sum $K_l(\balp; Z, H)$ by putting
\begin{align}\label{2.1}
    K_l(\balp; Z,H ) = \sum_{|h| \le H} \sum_{|z| \le Z} e(h \alp^{(1)} + 2hz \alp^{(2)} 
+ \ldots + lhz^{l-1}\alp^{(l)}),
\end{align}
and we write
\begin{align*}
    f_l(\balp; X) = \sum_{|x| \le X} e(\alp^{(1)} x + \alp^{(2)} x^2 + \ldots 
+ \alp^{(l)} x^l).
\end{align*}
Then, with the standard notation associated with Vinogradov's mean value theorem in 
mind, we put
\begin{align*}
    J_{s,l}(X) = \oint |f_l(\balp; X)|^{2s} \d \balp.
\end{align*}
We note that the main conjecture associated with Vinogradov's mean value theorem is 
now known to hold for all degrees. This is classical when $l=2$, it is a consequence of 
work of the second author \cite{Woo2016} for degree $l=3$, and for degrees exceeding 
three it follows from the work of Bourgain, Demeter and Guth, and of the second author 
(see \cite[Theorem~1.1]{BDG2016} and \cite[Corollary~1.3]{Woo2017}). Thus, for all 
$\sig > 0$ one has
\begin{align}\label{2.2}
    J_{\sig, l}(X) \ll X^{\eps}(X^{\sig} + X^{2\sig - l(l+1)/2}).
\end{align}
For future reference, we record the trivial inequality
\begin{align}\label{2.3}
    | a_1 \cdots a_n| \le |a_1|^n + \ldots + |a_n|^n,
\end{align}
which is valid for all $a_1, \ldots, a_n \in \C$.\par

We begin by bounding the mean value $M^*_{2}(X)$. 
\begin{lemma}\label{L2.1}
    Let $X$, $Z$ and $H$ be large real numbers. Then one has
    \begin{align}\label{2.4}
        \oint |f_2(\balp; X)K_2(\balp; Z,H)|^2 \d \balp \ll (XHZ)^\eps (HZ^2+XZ^2+XZH).
    \end{align}
\end{lemma}

\begin{proof}
    Upon considering the underlying system of equations, we see that the mean value on 
the left hand side of \eqref{2.4} is given by the number of integer solutions of the system 
of equations
	\begin{equation}\label{2.5}
	\begin{aligned}
        x_1^2-x_2^2 &= 2(h_1 z_1 - h_2 z_2)\\
        x_1-x_2 &= h_1-h_2,
    \end{aligned}
	\end{equation}
    with $|x_i| \le X$, $|h_i| \le H$ and $|z_i|\le Z$ for $i=1,2$. The second of these 
equations permits the substitution $h_2=h_1-x_1+x_2$ into the first, whence
    \begin{align*}
        (x_1-x_2)(x_1+x_2-2z_2)=2h_1(z_1-z_2).
    \end{align*}
    Suppose first that $h_1(z_1-z_2)$ is non-zero. Then for each of the $O(HZ^2)$ 
possible choices for $h_1$, $z_1$ and $z_2$ fixing the latter integer in such a manner, an 
elementary divisor function estimate shows there to be $O((HZ)^\eps)$ possible choices 
for the integers $x_1-x_2$ and $x_1+x_2-2z_2$, and hence also for $x_1$ and $x_2$. 
These choices also fix $h_2=h_1-x_1+x_2$, so we see that there are 
$O(H^{1+\eps}Z^{2+\eps})$ solutions of this first type. Meanwhile, if $h_1(z_1-z_2)=0$, 
then $h_1=0$ or $z_1=z_2$, and at the same time either $x_1=x_2$ or $x_1=2z_2-x_2$. 
In any case, therefore, each of the $O(XZ)$ possible choices for $z_2$ and $x_2$ 
determine $x_1$ and either $h_1$ or $z_1$. Since there are $O(Z+H)$ possible choices 
left by this constraint for the latter, and $h_2$ is again fixed by these choices just as 
before, we find that there are $O(XZ(Z+H))$ solutions of this second type. The conclusion 
of the lemma follows by summing the contributions from both types of solutions.
\end{proof}

Upon taking $X=H=Z$ in Lemma~\ref{L2.1}, we conclude that 
$M_{2}^*(X) \ll X^{3+\eps}$, which establishes \eqref{1.6} for $l=2$. We remark also that the system \eqref{2.5} can be interpreted as being of Vinogradov shape of degree two by means of the substitution $h_i = u_i-v_i$ and $z_i = u_i+v_i$ for $i=1,2$. Viewed in this way, Lemma~\ref{L2.1} amounts to no more than a rephrasing of the classical elementary proof of the quadratic case in Vinogradov's mean value theorem. \par

We now initiate the proof of Theorem~\ref{T1.3}, assuming the hypotheses of its 
statement. For $l \ge 2$, let
\begin{align*}
    g_l(\balp; X)=\sum_{|x| \le X}e(\alp^{(2)}x^2+\alp^{(3)}x^3+\ldots +\alp^{(l)}x^l).
\end{align*}
Define $\balp^{(l)}=(\alp_i^{(l)})_{1\le i\le r_l}$ for $2\le l\le k$. When $1\le i\le r_2$, write $\balp_i = (\alp_i^{(l)})$, where $l$ runs over all values for which $r_l \ge i$. We then put
\begin{align}\label{2.6}
    \gam_j^{(l)} = \sum_{i=1}^{r_l} c_{i,j}^{(l)} \alp_i^{(l)} \qquad (1 \le j \le s).
\end{align}
Also, set $\bgam_{j} = (\gam_j^{(l)})_{2 \le l \le k}$ for $1 \le j \le s$ and $\bgam^{(l)} = (\gam_j^{(l)})_{1 \le j \le s}$ for $2\le l\le k$, and put $\bgam = (\bgam_1, \dots, \bgam_s) = (\bgam^{(2)}, \dots, \bgam^{(k)})^T$. Then by orthogonality we have
\begin{align*}
    I_{s,\bk}^{\bv,\bu}(X) = \oint \prod_{j=1}^s |g_k(\bgam_j;X)|^2 \d \balp.
\end{align*}

\par Set $t_0=2$, and for a set of positive integers $t_1, \ldots, t_n$ to be fixed later take 
\begin{equation}\label{2.7}
	s_0 = t_0 u_0 + t_1 u_1 + \ldots + t_n u_n.
\end{equation}
Thus, on recalling (\ref{1.7}), we see in particular that
\[
	s_0 \ge u_0+u_1+\ldots +u_n=r_2.
\] 
Further, let $\calI$ denote the set of all integral $r_2$-tuples 
\[
	\left(j_{m,1},\ldots ,j_{m,u_m}\right)_{0 \le m \le n}
\]
with pairwise distinct entries $j_{m,h} \in\{1, \ldots, s\}$, and put 
\begin{align}\label{2.8}
	\calG_{\bt, \bk}^{\bv, \bu}(X) = \max_{\bj \in \calI}\oint \prod_{m=0}^n\prod_{h=1}^{u_m} |g_k(\bgam_{j_{m,h}}; X)|^{2s_0/r_2} \d \balp.
\end{align}
We can bound $I_{s,\bk}^{\bv,\bu}(X)$ in terms of $\calG_{\bt, \bk}^{\bv, \bu}(X)$. In 
particular, this will allow us to concentrate on the case when $s=s_0$. 

\begin{lemma}\label{L2.2}
	For any fixed choice of the positive integers $t_1,\ldots ,t_n$, we have the bounds
	\begin{align*}
		I_{s,\bk}^{\bv,\bu}(X) \ll \begin{cases}(\calG_{\bt, \bk}^{\bv, \bu}(X))^{s/s_0} &\text{when $r_2\le s\le s_0$,}\\
		X^{2s-2s_0} \calG_{\bt, \bk}^{\bv, \bu}(X) & \text{when $s > s_0$.} \end{cases}	
	\end{align*}
\end{lemma}

\begin{proof}
	When $s>s_0$, the trivial bound $g_k(\bgam_j;X)=O(X)$ delivers the estimate
	\begin{align*}
		I_{s,\bk}^{\bv,\bu}(X) \ll X^{2s-2s_0}\oint \prod_{j=1}^{s_0} 
|g_k(\bgam_j;X)|^2 \d \balp,
	\end{align*}
	and the conclusion of the lemma follows in this case from \eqref{2.3}. Suppose now that $r_2 \le s \le s_0$. Then from \eqref{2.3} and an application of H\"older's inequality, we find that
	\begin{align*}
		I_{s,\bk}^{\bv,\bu}(X) &\ll \max_{\bj\in \calI} \oint \prod_{m=0}^n\prod_{h=1}^{u_m} |g_k(\bgam_{j_{m,h}}; X)|^{2s/r_2} \d\balp\\
		& \ll \biggl( \max_{\bj\in \calI} \oint \prod_{m=0}^n\prod_{h=1}^{u_m} |g_k(\bgam_{j_{m,h}}; X)|^{2s_0/r_2} \d \balp \biggr)^{s/s_0}.
	\end{align*}
	Thus the lemma is established in both cases. 
\end{proof}

Suppose that the maximum in \eqref{2.8} is assumed at the tuple $\bj \in \calI$, which we consider fixed for the remainder of the analysis. For $2\le l\le k$ and $1\le i\le r_l$, set $d_{i, w_{m-1}+h}^{(l)}=c_{i, j_{m,h}}^{(l)}$ when $1 \le h \le u_m$ and $1 \le m \le n$, and likewise $d_{i,w_n+h}^{(l)}=c_{i, j_{0,h}}^{(l)}$ when $1 \le h \le u_0$. We then have the coefficient matrices
\[
	D^{(l)} = \bigl(d_{i,j}^{(l)}\bigr)_{\substack{1 \le i \le r_l\\ 1 \le j \le r_2}}\quad (2\le l\le k).
\]
We define $\del_i^{(l)}$ via the relations $\bdel^{(l)} = (D^{(l)})^T \balp^{(l)}$ for $2 \le l \le k$, and put $\bdel_j = (\del_j^{(l)})_{2 \le l \le k}$ for $1 \le j \le r_2$. Here, we employ notational conventions analogous to those described in the sequel to 
\eqref{2.6}.\par

Write 
\begin{align*}
	G_{t_j, k_j}^{v_j, u_j}(X) = \oint \prod_{h=1}^{u_j} |g_{k_j}(\bdel_{w_{j-1}+h}; X)|^{2t_j} \d \balp \qquad (1 \le j \le n).
\end{align*}
Thus, in the case $n=1$ and $u_0=0$, we have $\calG_{\bt, k}^{v, u}(X) = G_{t, k}^{v, u}(X)$. 

\begin{lemma}\label{L2.3}
	One has
	\begin{align*}
		\calG_{\bt, \bk}^{\bv, \bu}(X) \ll X^{2u_0+\eps} \prod_{j=1}^n G_{t_j, k_j}^{v_j, u_j}(X).
	\end{align*}
\end{lemma}

\begin{proof}
	Recall \eqref{2.7} and \eqref{2.8}. For temporary notational convenience, we put $u_{n+1}=u_0$. Then, after possibly relabelling indices, we see from \eqref{2.3} that
	\begin{align}\label{2.9}
		\calG_{\bt, \bk}^{\bv, \bu}(X) \ll \oint \prod_{j=1}^{n+1}\prod_{h=1}^{u_j} |g_k(\bdel_{w_{j-1}+h}; X)|^{2t_j} \d \balp.
	\end{align}
	The desired conclusion now follows by essentially the same argument as in \cite[Theorem~2.1]{BP2017}. Recall that the coefficient matrices $C^{(l)}$ are highly non-singular. Consequently, the matrices $D^{(l)}$ underlying the mean value in \eqref{2.8} inherit that property. Upon considering the underlying Diophantine equations and applying elementary row operations, we may thus assume without loss of generality that the first $r_l \times r_l$ submatrix of each matrix $D^{(l)}$ is diagonal.\par
	
	Recall the definition of the parameter $r_l$ from \eqref{1.7}. Since $\bj \in \calI$ has $r_2$ entries, we see that the matrix $D^{(2)}$ is of square format and hence diagonal. Thus we have $\del_i^{(2)} = d_{i,i}^{(2)} \alp_i^{(2)}$ for $1 \le i \le w_n+u_0$. In particular, the entries of $\bdel_i$ with $1 \le i \le w_n$ are independent of all the variables $\alp_{w_n+i}^{(2)}$ with $1 \le i \le u_0$. We may therefore interpret $\balp$ as the ordered pair $(\balp_{n+1}^{\dagger}, \balp_{n+1}^*)$ with $\balp_{n+1}^{\dagger} = (\balp_i)_{1 \le i \le w_n}$ and $\balp_{n+1}^* = (\alp_{w_n+i}^{(2)})_{1 \le i \le u_0}$. In this notation we can write 
	\begin{align}\label{2.10}
		\oint \prod_{j=1}^{n+1}\prod_{h=1}^{u_j} |g_k(\bdel_{w_{j-1}+h}; X)|^{2t_j} \d \balp=\oint \calF_{n+1}(\balp_{n+1}^{\dagger})G_{n+1}(\balp_{n+1}^{\dagger}; X) \d \balp_{n+1}^{\dagger},
	\end{align}
	where 
	\begin{align*}
		\calF_{n+1}(\balp_{n+1}^{\dagger}) = \prod_{m=1}^n\prod_{i=1}^{u_m} |g_k(\bdel_{w_{m-1}+i}; X)|^{2t_m}
	\end{align*}
	and 
	\begin{align*}
		G_{n+1}(\balp_{n+1}^{\dagger}; X) = \oint \prod_{i=1}^{u_0} |g_k(\bdel_{w_n+i};X)|^4 \d \balp_{n+1}^*.
	\end{align*}
	The latter mean value counts integer solutions $\bx, \by \in [-X,X]^{2u_0}$ of the system 
	\begin{align*}
		x_{i1}^2 + x_{i2}^2 - y_{i1}^2 - y_{i2}^2 = 0 \qquad ( 1 \le i \le u_0),
	\end{align*}
	where each solution is counted with a unimodular weight depending on $\balp_{n+1}^{\dagger}$. It then follows from the triangle inequality and Hua's lemma that 
	\begin{align}\label{2.11}
		G_{n+1}(\balp_{n+1}^{\dagger};X) \le G_{n+1}(\boldsymbol 0; X)  \ll X^{2u_0+\eps}.
	\end{align}
	
	We now iterate this procedure for $j=n,n-1,\ldots, 1$. For each index $j$, we see from \eqref{1.7} that $r_l > w_{j-1}$ only when $l \le k_j$. Moreover, we have $r_l \ge w_j$ if $l \le k_j -1$, and $r_l = w_{j-1}+v_j$ if $l = k_j$. Since we had arranged for the first $r_l \times r_l$ submatrices of each $D^{(l)}$ to be diagonal, it follows that the entries of $\bdel_i$ with $1 \le i \le w_{j-1}$ are independent of all the variables $\alp_{w_{j-1}+i}^{(l)}$ with $1 \le i \le u_j$ and $ 2 \le l \le k_j-1$, and also of all $\alp_{w_{j-1}+i}^{(k_j)}$ with $1 \le i \le v_j$. Together, these latter groups of variables form the vectors $\balp_i$ with $w_{j-1}+ 1 \le i \le w_j$. Hence by a similar argument to that encountered before, we can write $\balp_{j+1}^{\dagger} = (\balp_j^{\dagger}, \balp_j^*)$, where $\balp_j^{\dagger} = (\balp_i)_{1 \le i \le w_{j-1}}$ and $\balp_j^* = (\balp_{w_{j-1}+i})_{1 \le i \le u_j}$, noting in particular that the vector $\balp_1^{\dagger}$ is empty. For $2 \le j \le n$ put 
	\begin{align*}
		\calF_{j}(\balp_{j}^{\dagger}) = \prod_{m=1}^{j-1}\prod_{i=1}^{u_m} |g_k(\bdel_{w_{m-1}+i}; X)|^{2t_m},
	\end{align*}
	and take $\calF_{1}(\balp_{1}^{\dagger})=1$. Also, let
	\begin{align*}
		G_j(\balp_{j}^{\dagger}; X) = \oint \prod_{i=1}^{u_j} |g_k(\bdel_{w_{j-1}+i};X)|^{2t_j} \d \balp_j^* \qquad (1 \le j \le n).
	\end{align*}	
	Note that this mean value counts integer solutions $|\bx|, |\by| \le X$ to the system 
	\begin{align*}
		\sum_{i=1}^{u_j} d_{w_{j-1}+i,m}^{(k_j)} \sum_{h=1}^{t_j}(x_{i,h}^{k_j} -y_{i,h}^{k_j} )&=0 \qquad (1 \le m \le v_j),\\
       \sum_{i=1}^{u_j}d_{w_{j-1}+i,m}^{(l)}\sum_{h=1}^{t_j}(x_{i,h}^l-y_{i,h}^l)&=0 \qquad ( \text{$1 \le m \le u_j$, $2 \le l \le k_j-1$}),
	\end{align*}
	where each solution is counted with a unimodular weight depending on $\balp_j^\dagger$. An application of the triangle inequality shows that $G_j(\balp_{j}^{\dagger}; X) \le G_j(\boldsymbol 0; X) = G_{t_j,k_j}^{v_j, u_j}(X)$. We thus deduce that for $1 \le j \le n$ we have 
	\begin{align*}
		\oint \calF_{j+1}(\balp_{j+1}^{\dagger}) \d \balp_{j+1}^{\dagger} &= \oint \calF_j(\balp_j^{\dagger}) G_j(\balp_{j}^{\dagger}; X) \d \balp_j^{\dagger} 
		\ll G_{t_j, k_j}^{v_j, u_j}(X)\oint \calF_j(\balp_j^{\dagger})\d \balp_j^{\dagger}, 
	\end{align*} 
	and upon iterating we find that 
	\begin{align*}
		\oint \calF_{n+1}(\balp_{n+1}^{\dagger}) \d \balp_{n+1}^{\dagger}  \ll \prod_{j=1}^n G_{t_j, k_j}^{v_j, u_j}(X).
	\end{align*} 	
	The conclusion of the lemma follows upon combining this bound with \eqref{2.9}, 
\eqref{2.10} and \eqref{2.11}.	
\end{proof}

\section{The underlying mean value}\label{S3}
From Lemma~\ref{L2.3} it is clear that the desired bound $\calG_{\bt,\bk}^{\bv,\bu}(X) \ll X^{K+\eps}$ will follow if we can show that $G_{t_j,k_j}^{v_j,u_j}(X) \ll X^{K_j+\eps}$ for $j=1,\ldots, n$. We thus proceed to establish the latter bound. In the discussion of Lemmata \ref{L3.1} and \ref{L3.2} that follows, it is expedient to drop all mention of the indices $j$ with $1\le j\le n$. Note also that in this situation, we have $r_k=v$ and $r_l=u$ for $2 \le l \le k-1$. We introduce variables $\balp^{(1)} \in [0,1)^u$ and define $D^{(1)}$ to be the $u \times u$ identity matrix. Set further $\bdel^{(1)} = \balp^{(1)}$, and extend our previous notational conventions surrounding the vector $\bdel$ so as to  incorporate $\bdel^{(1)}$ in the natural manner.\par

Next, we define
\begin{align}\label{3.1}
    H_{t,k}^{v, u}(X)=\oint \prod_{j=1}^u |f_k(\bdel_j; 2X)|^{2t}K_k(-\bdel_j;X,2tX)\d \balp.
\end{align}
We begin by establishing the bound contained in the following lemma.

\begin{lemma}\label{L3.1}
    One has $G_{t, k}^{v, u}(X) \ll X^{-u} H_{t,k}^{v, u}(X)$.
\end{lemma}

\begin{proof}
    Define $\ome_l$ to be 1 when $l=1$, and $0$ otherwise. We decompose the set $\{1,\ldots ,tu\}$ into the blocks $\calB_m=\{(m-1)t+1,\ldots ,mt\}$ for $1\le m\le u$. The mean value $G_{t, k}^{v, u}(X)$ counts the number of integral solutions of the system of equations
    \begin{align}\label{3.2}
        \sum_{m=1}^u d_{j,m}^{(l)} \xi_m^{(l)} = \ome_l h_j \qquad (1 \le j \le r_l, \, 1 \le l \le k),
    \end{align}
    where
    \begin{align*}
        \xi_m^{(l)}= \sum_{i\in \calB_m}(x_i^l-y_i^l) \qquad (1\le m\le u, \, 1 \le l \le k),
    \end{align*}
    with $-X \le x_i, y_i \le X$ for $1 \le i \le tu$ and $|h_j| \le 2tX$ for $1 \le j \le u$. Observe that in our current situation all coefficient matrices $D^{(l)}$ with $1 \le l \le k-1$ are of format $u \times u$. Just as in the proof of Lemma~\ref{L2.3}, we can therefore assume without loss of generality that the coefficients $d_{j,m}^{(l)}$ with $1 \le l \le k-1$ vanish except when $j=m$. Also, note that the constraints on the expressions $\xi_m^{(1)}$ for $(1\le m\le u)$ imposed by the linear equations in \eqref{3.2} are void, since the ranges for the new variables $h_j$ automatically accommodate all possible values for $\xi_m^{(1)}$within \eqref{3.2}.\par

    We now consider the effect of shifting every variable with index in a given block $\calB_m$ by an integer $z_m$ with $|z_m| \le X$. By the binomial theorem, for any family of shifts $\bz$, one finds that $(\bx, \by)$ is a solution of \eqref{3.2} if and only if it is also a solution of the system
    \begin{align*}
        \sum_{m=1}^u d_{j,m}^{(l)} \zeta_m^{(l)} = \sum_{m=1}^u d_{j,m}^{(l)} lh_m z_m^{l-1} \qquad (1 \le j \le r_l,\, 1 \le l \le k),
    \end{align*}
    where
    \begin{align*}
        \zeta_m^{(l)} = \sum_{i \in \calB_m} \left( (x_i+z_m)^l -(y_i+z_m)^l\right)\qquad (1\le m\le u, \, 1 \le l \le k).
    \end{align*}
    Thus, for each fixed integer $u$-tuple $\bz$ with $|z_m| \le X$ ($1 \le m \le u$), the mean value $G_{t, k}^{v, u}(X)$ is bounded above by the number of integral solutions of the system
    \begin{align*}
        \sum_{m=1}^u d_{j, m}^{(l)}\sum_{i \in \calB_m} (v_i^l-w_i^l)= \sum_{m=1}^u d_{j,m}^{(l)} l h_m z_m^{l-1} \qquad (1 \le j \le r_l,\, 1 \le l \le k)
    \end{align*}
    with $|\bv|, |\bw| \le 2X$ and $|\bh| \le 2tX$. On applying orthogonality and averaging over all possible choices for $\bz$, we therefore infer that
    \begin{align*}
        G_{t, k}^{v, u}(X) \ll X^{-u} \sum_{|\bz| \le X} \oint \prod_{m=1}^u |f_k(\bdel_m; 2X)|^{2t} \frk(-\bdel_m; z_m) \d \balp,
    \end{align*}
    where
    \begin{align*}
        \frk(\balp; z) = \sum_{|h| \le 2tX} e(h \alp^{(1)} + 2hz \alp^{(2)} + \ldots + khz^{k-1}\alp^{(k)}).
    \end{align*}
    The proof of the lemma is completed by reference to \eqref{2.1} and \eqref{3.1}.
\end{proof}

We can now turn to the task of estimating $H_{t,k}^{v, u}(X)$.  We will do this in somewhat wider generality than is required for the proofs of Theorems \ref{T1.3} and \ref{T1.5}. This will allow us to prove the unconditional results adumbrated in the introduction.

When $l \ge 2$ is an integer and $\sig \ge m \ge 0$, denote by $M_{l, \sig, m}(X)$ the mean value 
\begin{align}\label{3.3}
	M_{l,\sig,m}(X) = \oint |f_l(\balp; X)|^{2\sig-2m} K_l(\balp; X,X)^m \d \balp.
\end{align}
When $\sig$ and $m$ are integers, this mean value counts the number of integer tuples $\bx,\by \in [-X,X]^{\sig-m}$ and $\bz,\bh \in [-X,X]^m$ satisfying
\begin{align*}
	\sum_{i=1}^{\sig-m}(x_i^j-y_i^j)+\sum_{i=1}^m jh_iz_i^{j-1}=0\qquad (1\le j\le l).
\end{align*}
In particular, we have $M_l^*(X) = M_{l, \frac12l(l+1),2}(X)$. The main conjecture for mean values of the shape \eqref{3.3} states that for all $\sig \ge m$ one should have 
\begin{align}\label{3.4}
	M_{l,\sig,m}(X) \ll X^{\eps}(X^{\sig} + X^{2\sig - l(l+1)/2}).
\end{align}
Note that the case when $m=0$ corresponds to Vinogradov's mean value theorem, and in this case the bound \eqref{3.4} is known (see equation \eqref{2.2} above).\par

Suppose that $\sig$ and $m$ are integers with $2 \le m \le \sig$ and 
\begin{align}\label{3.5}
	m |(\sig - \textstyle{\frac12}k(k-1)).
\end{align} 
We now choose 
\begin{align}\label{3.6}
	t = K/u + (\sig - \textstyle{\frac12}k(k-1))/m, 
\end{align}
so that at the critical point $\sig = \frac12 k(k-1)$ we have $tu=K$. Note also that by \eqref{3.5} as well as the definition of $K$ in \eqref{1.4}, the quantity $t$ is indeed an integer whenever $u|kv$. 

\begin{lemma}\label{L3.2}
	Let $\sig$ and $m$ be integers with $2 \le m \le \sig$ and satisfying the conditions 
$2|m$ and $m|(\sig-\frac12k(k-1))$. Assume also that $u \ge \frac{m}{m-1} v$ and $u|kv$. Then we have 
	\begin{align*}
		H_{t,k}^{v,u}(X)\ll X^{K+u+\eps}\left( \frac{M_{k-1,\sig,m}(2tX)} {X^{k(k-1)/2}}\right)^{u/m}.
	\end{align*}
\end{lemma}

\begin{proof}
	Set 
	\begin{equation}\label{3.7}
		\frG_1(\bdel) = \prod_{j=1}^v |f_k(\bdel_j; 2X)|^{k(k+1)}
 	\end{equation}
 	and 
	\[
		\frG_2(\bdel)=\prod_{j=v+1}^u\left|f_k(\bdel_j;2X)^{2t-k(k+1)v/u}K_k(\bdel_j; X, 2tX)\right|^{u/(u-v)}. 
	\]
	Then it follows from (\ref{3.1}) via \eqref{2.3} that, after possibly relabelling variables, we have 
	\begin{align}\label{3.8}
		H_{t,k}^{v, u}(X) \ll \oint \frG_{1}(\bdel) \frG_{2}(\bdel) \d \balp.
	\end{align}
	Recall now that we had arranged for the coefficient matrices $D^{(1)},\ldots,D^{(k-1)}$ to be diagonal. Consequently, the variables $\bdel_1,\ldots, \bdel_v$ are independent of those $\alp_i^{(l)}$ having $1 \le l \le k-1$ and $v+1 \le i \le u$. Then, by setting $\bfeta_1 = (\balp_i)_{1 \le i \le v}$ and $\bfeta_2 = (\balp_i)_{v+1 \le i \le u}$, it follows that $\bfeta_1$ fully determines $\bdel_1, \ldots, \bdel_v$, and $\bfeta_1$ and $\bfeta_2$ together completely determine all entries of $\bdel$. On recalling (\ref{3.7}), we may thus rewrite the integral on the right hand side of \eqref{3.8} to obtain the bound
	\begin{align}\label{3.9}
		H_{t,k}^{v,u}(X)\ll \oint \frH_1(\bfeta_1) \frH_2(\bfeta_1) \d \bfeta_1,
	\end{align}	
	where $\frH_1(\bfeta_1)= \frG_1(\bdel)$ and 
	\[
		\frH_2(\bfeta_1)=\oint \frG_2(\bdel) \d \bfeta_2.
	\]

	\par Define 
	\[
		U_1(\bfeta_1)= \oint \prod_{j=v+1}^u |f_k(\bdel_j; 2X)|^{k(k-1)} \d \bfeta_2
	\]
	and
	\[
		U_2(\bfeta_1)=\oint \prod_{j=v+1}^u|f_k(\bdel_j;2X)^{2\sig-2m}K_k(\bdel_j;X, 2tX)^m| \d \bfeta_2.
	\]
	Also, write
	\[
		\ome=\frac{u}{m(u-v)},
	\]
	and note that, as a consequence of \eqref{1.4} and \eqref{3.6}, one has
	\begin{align*}
		(2\sig -2m)\ome + k(k-1) (1-\ome) &= \ome (2mt - 2Km/u - 2m) + k(k-1)\\
		&= \frac{2tu - (k(k-1)u+2kv-2u)  - 2u}{u-v} + k(k-1)\\
		&= \left( 2t-k(k+1)v/u\right) \frac{u}{u-v}.
	\end{align*}
	Then, since $m \ge 2$ and $u \ge \frac{m}{m-1}v$, it follows via H\"older's inequality that 
	\begin{align}\label{3.10}
		\frH_2(\bfeta_1) \ll U_1(\bfeta_1)^{1-\ome} U_2(\bfeta_1)^\ome .
	\end{align}
	
	Since $m$ is an even integer, it follows by standard orthogonality considerations that $U_1(\bfeta_1)$ and $U_2(\bfeta_1)$ count solutions to their respective associated systems of equations with degrees $1, \ldots, k-1$, with each solution being counted with a unimodular weight depending on $\bfeta_1$. It thus follows from the triangle inequality that $U_i(\bfeta_1) \le U_i(\boldsymbol 0)$ for $i =1,2$. Using the fact that the coefficient matrices $D^{(l)}$ with $1 \le l \le k-1$ are all diagonal, and recalling \eqref{2.2}, we thus discern that
	\begin{equation}\label{3.11}
     	U_1(\bfeta_1) \ll \oint \prod_{j=v+1}^u |f_{k-1}(\balp_j; 2X)|^{k(k-1)} \d \bfeta_2 \ll X^{\frac12k(k-1)(u-v)+\eps}.
	\end{equation}
	By an analogous chain of reasoning, we derive from the definition (\ref{3.3}) of $M_{l,\sig,m}(X)$  and a consideration of the underlying system of equations the corresponding bound
	\begin{align}\label{3.12}
    	U_2(\bfeta_1)&\ll \oint \prod_{j=v+1}^u|f_{k-1}(\balp_j;2X)^{2\sig-2m} K_{k-1}(\balp_j;X, 2tX)^{m}|\d \bfeta_2 \nonumber\\
    	& \ll \big(M_{k-1, \sig, m}(2tX)\big)^{u-v}.
	\end{align}
	Thus, from \eqref{3.10}, \eqref{3.11} and \eqref{3.12} we have 
	\begin{align*}
		\frH_2(\bfeta_1) &\ll X^{\frac12 k(k-1)(u-v)(1-\ome)+\eps}M_{k-1, \sig, m}(2tX)^{(u-v)\ome}\\
		&\ll X^{\frac12 k(k-1)(u-v) + \eps} \left(\frac{M_{k-1, \sig, m}(2tX)}{X^{k(k-1)/2}}\right)^{u/m}.
	\end{align*}

	\par	At this stage in our argument, we discern from \eqref{3.9} that 
	\begin{align}\label{3.13}
		H_{t,k}^{v, u}(X) &\ll X^{\frac12 k(k-1)(u-v) + \eps}\left(\frac{M_{k-1,\sig,m}(2tX)}{X^{k(k-1)/2}}\right)^{u/m}\oint \frH_1(\bfeta_1)\d \bfeta_1.
	\end{align}	
	Recall that $\frH_1(\bfeta_1)=\frG_1(\bdel)$, where $\frG_1(\bdel)$ is defined by 
(\ref{3.7}). Since the first $v \times v$ minors of the coefficient matrices $D^{(l)}$  for 
$1 \le l \le k$ are now diagonal, we deduce from \eqref{2.2} that
	\begin{align}\label{3.14}
		\oint \frH_1(\bfeta_1) \d \bfeta_1 \ll (J_{k(k+1)/2, k}(2X))^v \ll X^{\frac12k(k+1)v+\eps}.
	\end{align}	
	Finally, on substituting \eqref{3.14} into \eqref{3.13} and recalling \eqref{1.4}, we conclude that
	\[
		H_{t,k}^{v, u}(X) \ll X^{K+u+\eps}\left(\frac{M_{k-1, \sig, m}(2tX)}{X^{k(k-1)/2}}\right)^{u/m}.
	\]
	This completes the proof of the lemma.
\end{proof}

We now resume the practice of appending the suffix $j$ to the parameters $k$, $u$, $v$ 
and $K$ that we temporarily abandoned during the discussion of Lemmata \ref{L3.1} and 
\ref{L3.2}. We assume, moreover, that $\sig_j$ and $m_j$ are integers with 
$2\le m_j\le \sig_j$ and
\[
m_j|(\sig_j-\tfrac{1}{2}k_j(k_j-1)).
\]
In accordance with \eqref{3.6}, we now fix the parameters $t_j$ by taking
\begin{align}\label{3.15}
	t_j=K_j/u_j+(\sigma_j-\tfrac{1}{2}k_j(k_j-1))/m_j\quad (1\le j\le n).
\end{align}
Hence, whenever $u_j|k_jv_j$, the quantity $t_j$ is an integer. With these natural numbers $t_j$ defined thus, we recall the definition of $s_0$ from \eqref{2.7}. We are now equipped to provide an unconditional version of Theorem~\ref{T1.3}

\begin{theorem}\label{T3.3}
	Suppose that $k_1>\ldots >k_n\ge 3$. Assume further that $\bu,\bv,\bm,\bsig\in \N^n$ satisfy the relations
	\[
		2|m_j,\quad 2\le m_j\le \sig_j,\quad u_j \ge \frac{m_j}{m_j-1} v_j,\quad u_j | k_j v_j\quad \text{and}\quad m_j|(\sig_j-\tfrac{1}{2}k_j(k_j-1))
	\]
 	for $1 \le j \le n$. Let $u_0$ be a non-negative integer. Then for any $\eps>0$, one has 
	\[
		I_{s_0,\bk}^{\bv,\bu}(X) \ll X^{K+\eps} \prod_{j=1}^n \left( \frac{M_{k_j-1,\sig_j,m_j}(2t_jX)}{X^{k_j(k_j-1)/2}} \right)^{u_j/m_j}.
	\]
\end{theorem}

\begin{proof} 
	We apply Lemma \ref{L2.2} with $s=s_0$, followed by Lemmata \ref{L2.3}, \ref{L3.1} and \ref{L3.2}. This shows that
	\[
		I_{s_0,\bk}^{\bv,\bu}(X) \ll X^{2u_0+\eps}\prod_{j=1}^n X^{K_j}\left( \frac{M_{k_j-1,\sig_j,m_j}(2t_jX)}{X^{k_j(k_j-1)/2}} \right)^{u_j/m_j},
	\]
	and the proof is complete upon reference to \eqref{1.9}.
\end{proof}

We can now complete the proof of Theorem~\ref{T1.3}. To this end, we choose $m_j=2$ and $\sig_j=\frac12k_j(k_j-1)$ for $1\le j\le n$. With this choice of parameters the hypotheses of Theorem \ref{T3.3} are satisfied whenever $\bk, \bv$ and $\bu$ are in accordance with the conditions of Theorem \ref{T1.3}, and moreover the conjectural bound $M_{k_j-1,\sig_j,2}(2t_jX)\ll X^{\sig_j+\eps}$ is then tantamount to both \eqref{1.6} and \eqref{3.4}. Thus, in the case $s=s_0$ the desired conclusion is an immediate consequence of the conclusion of Theorem \ref{T3.3}, and for general values of $s$ it follows in like manner upon utilising the additional flexibility offered by Lemma \ref{L2.2}.

\section{The Hardy-Littlewood method}\label{S4}
We can now initiate the derivation of Theorem~\ref{T1.5} from the mean value estimate of Theorem~\ref{T1.3}. We shall prove the following rather more general result. 

\begin{theorem}\label{T4.1}
	Suppose that $k_1 > \ldots > k_n \ge 3$. Suppose further that $\bu, \bv, \bm, \bsig \in \N^n$ lie in the respective ranges
	\[
		2\le m_j\le \sig_j,\quad \sig_j \ge \tfrac12 k_j(k_j-1) \quad \text{and}\quad u_j\ge \frac{m_j}{m_j-1} v_j,
	\]
	and satisfy the divisibility conditions
	\[
		2|m_j,\quad u_j|k_j v_j\quad \text{and}\quad m_j | (\sig_j-\tfrac{1}{2}k_j(k_j-1)),
	\]
	for $1 \le j \le n$. Assume, moreover, that
	\begin{align}\label{4.0}
		M_{k_j-1,\sig_j,m_j}(X)\ll X^{2\sig_j-\tfrac{1}{2}k_j(k_j-1)+\eps} \quad (1\le j\le n).
	\end{align}
	Let $u_0$ be a non-negative integer, put $t_0=2$ and define $t_j$ via \eqref{3.15} for $1 \le j \le n$. Set $s_0=t_0u_0 + \ldots + t_n u_n$, suppose that $s \ge 2s_0+1$ and write $K=2u_0+K_1+\ldots +K_n$ for the total degree of the system as usual. Then the asymptotic formula 
	\[
		N_{s,\bk}^{\bv,\bu}(X)=(\calC+o(1))X^{s-K}
	\]
	holds with $\calC \ge 0$. If, furthermore, the system \eqref{1.10} has non-singular solutions in $\R$ as well as in the fields $\Q_p$ for all $p$, then the constant $\calC$ is positive. 
\end{theorem}

Note that Theorem~\ref{T1.5} follows from the special case of Theorem \ref{T4.1} in 
which $m_j=2$ and $\sig_j = \frac12 k_j (k_j-1)$ for $1 \le j \le n$. 

We make use of the notation introduced in \S\ref{S2}, and recall in particular \eqref{2.6} 
and its sequel. From now on we will set $r =r_2+\ldots +r_k$ and $w=u_0+\ldots +u_n$, 
so that $w = r_2 = w_n+u_0$. Also, we will assume throughout that $\bk$, $\bv$, $\bu$, 
$\bsig$, $\bm$ satisfy the hypotheses of the statement of Theorem~\ref{T4.1}.\par

When $\frB \subseteq [0,1)^{r}$ is a measurable set, put
\begin{equation}\label{4.1}
    N_{s,\bk}^{\bv,\bu}(X; \frB) = \int_{\frB} \prod_{j=1}^s g_k(\bgam_j;X) \d \balp.
\end{equation}
Our Hardy--Littlewood dissection is defined as follows. When $Y$ and $Q$ are parameters with $1 \le Q \le Y$, we take the major arcs $\frM_Y=\frM_Y(Q)$ to be the union of the boxes
\begin{align}\label{4.2}
    \frM_Y(q,\ba)=\left\{ \balp\in [0,1)^r:| \alp_j^{(l)}-a_j^{(l)}/q| \le QY^{-l}\quad (1 \le j \le r_l,\, 2 \le l \le k)\right\},
\end{align}
with $0\le \ba\le q\le Q$ and $(q,\ba) = 1$. The corresponding set of minor arcs $\frm_Y=\frm_Y(Q)$ is defined by putting $\frm_Y(Q)= [0,1)^r\setminus \frM_Y(Q)$. Unless indicated otherwise, we fix $Y=X$ and $Q=X^{1/(6r)}$, and abbreviate $\frM_X$ to 
$\frM$ and $\frm_X$ to $\frm$.\par

We require certain auxiliary functions in order to analyse the contribution of the major arcs $N_{s,\bk}^{\bv,\bu}(X;\frM)$. Write
\begin{align*}
    S_k(q, \ba) = \sum_{x=1}^q e((a^{(2)} x^2 + \ldots + a^{(k)} x^k)/q),
\end{align*}
and recall that the argument of \cite[Theorem~7.1]{Vau} gives
\begin{align}\label{4.3}
    S_k(q, \ba) \ll (q, \ba)^{1/k}q^{1-1/k+\eps}.
\end{align}
Further, set
\begin{align*}
    v_k(\bbet; X) = \int_{-X}^X e(\bet^{(2)}z^2 + \ldots + \bet^{(k)} z^k) \d z,
\end{align*}
and recall from the arguments of \cite[Theorem~7.3]{Vau} the estimate
\begin{align}\label{4.4}
    v_k(\bbet; X) \ll X \left(1+|\bet^{(2)}|X^2+ \ldots + 
|\bet^{(k)}|X^k\right)^{\hskip-.5mm-1/k}.
\end{align}
We put
\begin{align}\label{4.5}
    \Lam_j^{(l)} = \sum_{i=1}^{r_l} c_{i,j}^{(l)} a_i^{(l)} \quad \text{and} \quad 
\tet_j^{(l)}=\sum_{i=1}^{r_l}c_{i,j}^{(l)}\bet_i^{(l)} \qquad (1 \le j \le s,\, 2 \le l \le k).
\end{align}
Following the same convention regarding vector notation as we applied for $\bgam$ in 
\eqref{2.6} and its sequel, we have $\btet = \bgam - \bLam/q$. Then as a consequence 
of \cite[Theorem~7.2]{Vau}, we find that when $\balp = \ba/q+\bbet \in \frM$, one has
\begin{align}\label{4.6}
    g_k(\bgam_j; X) = q^{-1} S_k(q, \bLam_j) v_k(\btet_j; X) + O(Q^2).
\end{align}

\par Finally, define
\begin{align}\label{4.7}
    \frS(Q) = \sum_{q \le Q} \sum_{\substack{1 \le \ba \le q \\ (q, \ba)=1}} 
\prod_{j=1}^s q^{-1}S_k(q, \bLam_j)
\end{align}    
and
\begin{align*}    
     \frJ_X(Q) = \int_{\calI(X,Q)} \prod_{j=1}^s v_k(\btet_j; X) \d \bbet,
\end{align*}
where
\begin{align*}
    \calI(X, Q) = \prod_{l=2}^k [-QX^{-l}, QX^{-l}]^{r_l}.
\end{align*}
The preliminary conclusion of our major arcs analysis is summarised in the following 
lemma.

\begin{lemma}\label{L4.1}
    There is a positive number $\ome$ for which
    \begin{align*}
        N_{s, \bk}^{\bv, \bu}(X;\frM) = X^{s-K}\frS(Q) \frJ_1(Q) + O(X^{s-K-\ome}).
    \end{align*}
\end{lemma}

\begin{proof}
     Since $\vol(\frM)\ll Q^{2r+1}X^{-K}$, it follows from \eqref{4.6} that
    \begin{align}\label{4.8}
        N_{s, \bk}^{\bv, \bu}(X; \frM) = \frS(Q) \frJ_X(Q) + O(X^{s-K-1}Q^{2r+3}).
    \end{align}
    Furthermore, by a change of variables we see that $\frJ_X(Q) = X^{s-K} \frJ_1(Q)$. 
The conclusion of the lemma therefore follows from our choice $Q=X^{1/(6r)}$.
\end{proof}

In order to address the contribution of the minor arcs, we need the following Weyl-type 
estimate.
\begin{lemma}\label{L4.2}
    Suppose that $\balp \in \frm$. There exists $\tau>0$ such that for each $w$-tuple 
$(j_1,\ldots, j_w)$ of distinct indices there exists an index $i$ with $1\le i\le w$ for which 
one has
    \begin{align*}
        |g_k(\bgam_{j_i}; X)| \le XQ^{-\tau}.
    \end{align*}
\end{lemma}
\begin{proof}
    This is the content of \cite[Lemma~3.1]{BP2017}. Note that the minor arcs in our 
setting are a subset of the minor arcs defined in the context of that lemma. 
\end{proof}

We now complete the analysis of the minor arcs for Theorem~\ref{T4.1}.
\begin{lemma}\label{L4.3}
    Assume the hypotheses of Theorem~\ref{T4.1}. Then there is a positive number 
$\ome$ for which $N_{s, \bk}^{\bv, \bu}(X;\frm) \ll X^{s-K-\ome}$.
\end{lemma}

\begin{proof}
    Given a measurable set $\frB \subseteq [0,1)^r$, we write
    \begin{align*}
        N^*(X; \frB) = \int_{\frB}\prod_{j=1}^{2s_0+1}|g_k(\bgam_j; X)| \d \balp.
    \end{align*}
	We begin by estimating the last $s-(2s_0+1)$ exponential sums in the product \eqref{4.1} trivially, so that
    \begin{equation}\label{4.9}
        N_{s,\bk}^{\bv,\bu}(X;\frm) \ll X^{s-(2s_0+1)}N^*(X;\frm) .
    \end{equation}
    For $1 \le i \le w$ and $\tau>0$ sufficiently small, let $\frm^{(i)}$ denote the set of $\balp \in [0,1)^r$ for which $|g_k(\bgam_i; X)| \le XQ^{-\tau}$. In view of \eqref{2.3}, we can identify a subset of indices $\calJ_i \subseteq \{1, \dots, 2s_0+1\} \setminus \{ i \}$ with $\card(\calJ_i)=s_0$ for which
    \begin{align*}
	    N^*(X;\frm^{(i)})\ll XQ^{-\tau}\oint \prod_{j\in \calJ_i}|g_k(\bgam_j;X)|^2\d\balp.
    \end{align*} 
    Write $C_i^{(l)}$ for the submatrix of $C^{(l)}$ having columns indexed by $\calJ_i$. The condition that the coefficient matrices $C^{(l)}$ be highly non-singular implies that the submatrices $C_i^{(l)}$ of $C^{(l)}$ are also highly non-singular. Thus, by orthogonality, we see from the definition \eqref{1.8} of the mean value $I_{s_0,\bk}^{\bv,\bu}(X)$ that  
    \begin{align*}
      N^*(X;\frm^{(i)})\ll XQ^{-\tau}I_{s_0,\bk}^{\bv,\bu}(X;C^{(2)}_i,\ldots,C^{(k)}_i).
    \end{align*}
    Consider a fixed $\balp \in \frm$. If $\tau$ has been chosen sufficiently small, Lemma~\ref{L4.2} ensures that we can find an index $j$ with $1 \le j \le w$ such that $\balp \in \frm^{(j)}$. Thus we see that we have the inclusion $\frm \subseteq \frm^{(1)} \cup \dots \cup \frm^{(w)}$, whence
    \begin{align}\label{4.10}
        N^*(X;\frm)\ll XQ^{-\tau}\max_{1 \le i \le w} I_{s_0,\bk}^{\bv,\bu}(X;C_i^{(2)}, \ldots, C_i^{(k)}).
    \end{align}

    \par Now recall that $Q=X^{1/(6r)}$. Note also that the hypotheses of Theorem \ref{T4.1} under which we are currently working permit the assumption of those of Theorem \ref{T3.3}. Thus, upon combining the estimate \eqref{4.10} with Theorem~\ref{T3.3}, inserting \eqref{4.0} and recalling \eqref{3.15}, we obtain the bound
	\begin{align*}
        N^*(X;\frm)&\ll XQ^{-\tau}X^{K+\eps}\prod_{j=1}^n \left(X^{2\sig_j-k_j(k_j-1)}\right)^{u_j/m_j}\\
        &\ll XQ^{-\tau}X^{K+\eps}\prod_{j=1}^n X^{2t_ju_j-2K_j}\\
        &\ll X^{2s_0+1-K}Q^{-\tau/2}.
    \end{align*}
	By substituting this estimate into \eqref{4.9}, we obtain the conclusion of the lemma.
\end{proof}

Upon combining the results of Lemmata \ref{L4.1} and \ref{L4.3}, we infer that for some $\omega>0$ one has the asymptotic formula
\begin{align}\label{4.11}
    N_{s,\bk}^{\bv,\bu}(X) = X^{s-K}\frS(Q) \frJ_1(Q) + O(X^{s-K-\ome}).
\end{align}
This completes our analysis of the minor arcs.

\section{Initial considerations for the singular series}\label{S5}
It remains to show that the singular series $\frS(Q)$ and singular integral $\frJ_1(Q)$ converge as $Q$ tends to infinity. We now put 
\begin{align*}
	\ell_j = \begin{cases}
		k_{i} & \text{when $w_{i-1}+1 \le j \le w_{i-1}+v_i$}, \\
		k_{i}-1 & \text{when $w_{i-1}+ v_i+1  \le j \le w_{i}$},\\
		2 & \text{when $w_n+1 \le j \le w$}.
	\end{cases}
\end{align*}
In this notation, the system under consideration can be viewed as a superposition of $w$ Vinogradov systems with respective degrees $\ell_j$, all missing the linear slice, and thus it follows from the definition \eqref{1.9} that the total degree of this system is
\begin{align*}
    K =  \sum_{j=1}^w \big( \textstyle{\frac12} \ell_j(\ell_j+1)-1 \big) .
\end{align*}
Throughout this and the next section, we work under the assumption that $s \ge 2K+1$. 

We first attend to the singular series. Put
\begin{align}\label{5.1}
    A(q) = q^{-s} \msum{q}{\ba} \prod_{j=1}^s S_k(q, \bLam_j).
\end{align}
By applying \eqref{2.3}, we find that for some choice of distinct indices $j_1, \ldots, j_w \in \{1, \ldots, s\}$ we have the asymptotic bound
\begin{align}\label{5.2}
    A(q)&\ll q^{2K-s} \max_{\substack{1 \le \ba \le q \\ (q, \ba) = 1}} \left(\prod_{i=1}^w |S_k(q, \bLam_{j_i})| \right)^{(s-2K)/w} A_1(q),
\end{align}
where
\begin{align*}
    A_1(q)=q^{-2K}\msum{q}{\ba}\prod_{i=1}^w
|S_k(q,\bLam_{j_i})|^{\ell_i(\ell_i+1)-2}.
\end{align*}
Note that both $A(q)$ and $A_1(q)$ are multiplicative in $q$. For this reason, the key to understanding the singular series is to maintain good control over the multiplicative quantity 
\begin{equation}\label{5.3}
	 B_1(q) = \sum_{d | q} A_1(d)
\end{equation}
as $q$ runs over the prime powers. 

Define $\tau_j$ by setting $\tau_j = \frac12\ell_j(\ell_j+1)-1$ for $ 1 \le j \le w$, and write $T_j = \tau_1 + \ldots + \tau_j$, so that $T_w = K$. For consistency we also set $T_0 = 0$. Now, adopting a notation similar to that of Section \ref{S2}, when $2 \le l \le k$ we write $D^{(l)}$ for the submatrices
\begin{align*}
    \big(d_{i, h}^{(l)}\big)_{\substack{1 \le i \le r_l\\ 1 \le h \le w}} = 
\big(c_{i, j_h}^{(l)}\big)_{\substack{1 \le i \le r_l\\ 1 \le h \le w}}
\end{align*}
of the coefficient matrices $C^{(l)}$ consisting of the columns indexed by 
$j_1, \ldots, j_w$. Note that the hypothesis that each $C^{(l)}$ is highly non-singular 
ensures that the same is true for each $D^{(l)}$. For $1 \le h \le w$ and $2\le l\le k$ we 
set $\Del_h^{(l)} = \Lam_{j_h}^{(l)}$, and we employ the same conventions regarding 
vector notation as in \eqref{4.5} and also \eqref{2.6} and its sequel. Thus, we write 
$\bDel_j = (\Del_j^{(l)})_{2 \le l \le k}$ and 
$\bDel^{(l)} = (\Del_j^{(l)})_{1 \le j \le w}$, so that 
\begin{align}\label{5.4}
	\bDel^{(l)} = (D^{(l)})^T \ba^{(l)} \qquad (2 \le l \le k).
\end{align}
In this notation, it follows from standard orthogonality relations that 
\begin{align*}
	q^{2K-r}B_1(q)=q^{-r}\sum_{1\le\ba\le q}\prod_{j=1}^w|S_k(q,\bDel_j)|^{2\tau_j}
\end{align*}	
counts the number of solutions $\bx, \by \in (\Z/q\Z)^K$ of the system of congruences
\begin{align}\label{5.5}
	\sum_{j=1}^w d_{i,j}^{(l)} \Bigg(\sum_{h=T_{j-1}+1}^{T_j} (x_h^l - y_h^l)\Bigg) 
\equiv 0 \pmod {q},
\end{align}
where $1 \le i \le r_l$ and $2 \le l \le k$.

Our first goal is to apply a procedure inspired by the proof of Theorem~2.1 in 
\cite{BP2017} in order to disentangle the congruences in \eqref{5.5}. This will enable us 
to  replace the sum $B_1(q)$ by a related expression in which for all indices $j$ the 
degree $k$ in the exponential sum $S_k(q, \bDel_j)$ is replaced by $\ell_j$. Since $\ell_j$ 
is typically smaller than $k$, we will reap the rewards of this preparatory step when the 
reduced degrees allow us to exert greater control on the size of the exponential sums in 
question. 

Given a $(k-1)$-tuple of variables $\xi^{(2)}, \ldots, \xi^{(k)}$, we adopt the convention 
that $\bxi^{[l]} = (\xi^{(2)}, \ldots, \xi^{(l)})$ for $2 \le l \le k$. Also, when 
$\bd = (d_2, \dots, d_k)$ is a coefficient vector, we abbreviate the vector 
$(d_2 \xi^{(2)}, \dots, d_k \xi^{(k)})$ to $\bd \bxi$, and we appropriate the notation 
$\bd^{[l]}$ and $(\bd \bxi)^{[l]}$ to denote the corresponding subvectors whose 
entries are indexed by $2 \le i \le l$. The following observation will play a part in our 
ensuing arguments.

\begin{lemma}\label{L5.1}
	Let $l$, $q$ and $t$ be natural numbers, with $2 \le l \le k-1$. Suppose that 
$d_2, \dots, d_k$ and $c_2, \dots, c_k$ are fixed integers, and put 
	\begin{align*}
		\Gam_q(\bd^{[l]}) = \prod_{j=2}^{l} (q, d_j).
	\end{align*} 
	Then for any fixed integers $a^{(l+1)}, \dots, a^{(k)}$ we have
	\begin{align*}
		\sum_{1 \le \ba^{[l]} \le q} |S_k(q, \bd \ba + \bc)|^{2t} \le \Gam_q(\bd^{[l]}) 
\sum_{1 \le \ba^{[l]} \le q} |S_l(q, \ba^{[l]})|^{2t}.
	\end{align*}
\end{lemma}

\begin{proof}
	By standard orthogonality relations, the sum 
	\begin{align}\label{5.6}
		T = q^{1-l}\sum_{1 \le \ba^{[l]} \le q} |S_k(q, \bd\ba + \bc)|^{2t} 
	\end{align}	
	counts solutions $\bx, \by \in (\Z/q\Z)^t$ of the system of congruences
	\begin{align}\label{5.7}
		d_j \sum_{i=1}^{t} (x_i^j-y_i^j) \equiv 0 \pmod {q} \qquad (2 \le j \le l),
	\end{align}
	where each solution is counted with a unimodular weight depending on the inert 
variables $a^{(l+1)}, \ldots, a^{(k)}$, together with the coefficients $\bd$ and $\bc$. 
Thus, by the triangle inequality, one finds that 
	\begin{align*}
		T \le q^{1-l} \sum_{1 \le \ba^{[l]} \le q} |S_l(q, (\bd \ba)^{[l]})|^{2t}. 
	\end{align*}
	We therefore discern that $T$ is bounded above by the number of solutions of 
\eqref{5.7} counted without weights, and hence by the number of solutions 
$\bx, \by \in (\Z/q\Z)^t$ of the system of congruences 
	\begin{align*}
		\sum_{i=1}^{t} (x_i^j - y_i^j) \equiv 0 \pmod{q/(q, d_j)} \qquad (2 \le j \le l).
	\end{align*}
	We interpret the latter as the number of solutions of the system
	\begin{align*}
		\sum_{i=1}^t (x_i^j-y_i^j) \equiv \frac{e_j q}{(q,d_j)} \pmod q \qquad 
(2 \le j \le l),
	\end{align*}
	with $\bx, \by \in (\Z/q\Z)^t$ and $1 \le e_j \le (q, d_j)$ for $2 \le j \le l$. Thus, by 
orthogonality and the triangle inequality, one sees that 
	\begin{align*}
		T &\le \sum_{\substack{1 \le e_j \le (q, d_j) \\ (2 \le j \le l)}} q^{1-l} 
\sum_{1 \le \ba^{[l]} \le q} |S_l(q, \ba^{[l]})|^{2t} e \left( - \sum_{j=2}^l 
\frac{e_ja^{(j)}}{(q, d_j)}\right)\\
		&\le \Gam_q(\bd^{[l]}) q^{1-l} \sum_{1 \le \ba^{[l]} \le q} 
|S_l(q, \ba^{[l]})|^{2t}.
	\end{align*} 
	The conclusion of the lemma is now immediate from \eqref{5.6}.
\end{proof}

We now define 
\begin{align*}
	B_1^*(q) = q^{-2K}\sum_{1 \le \ba \le q} \prod_{j=1}^w 
|S_{\ell_j}(q, \ba_j^{[\ell_j]})|^{2\tau_j}.
\end{align*}
The crucial bound for our analysis of the singular series is contained in the following 
lemma.

\begin{lemma}\label{L5.2}
	Let $q$ be a natural number, and suppose that the matrices $D^{(l)}$ are all highly 
non-singular. Then there exists a finite set of primes $\Ome(D)$ and a natural number 
$\calR(q) = \calR(q,D)$, both depending at most on the coefficient matrices $D^{(l)}$ 
and in the latter case also $q$, with the property that 
	\begin{align*}
		B_1(q) \le \calR(q) B_1^*(q).
	\end{align*}
	The constant $\calR(q)$ is bounded above uniformly in $q$, and one can take 
$\calR(q) = 1$ whenever $(q, p)=1$ for all $p \in \Ome(D)$.	
\end{lemma}

\begin{proof}
	Recall that $q^{2K-r}B_1(q)$ counts the number of solutions $\bx, \by \in (\Z/q\Z)^K$ 
of the system of congruences \eqref{5.5} for $1 \le j \le r_l$ and $2 \le l \le k$. Since 
$B_1(q)$ is a multiplicative function of $q$, it is apparent that it suffices to establish the 
conclusion of the lemma in the special case in which $q$ is a prime power, say $q=p^h$ 
for a given prime $p$. By applying suitable elementary row operations within the 
coefficient matrices $D^{(l)}$ for $2 \le l \le k$ that are invertible over $\Z/p^h\Z$, we 
may suppose without loss of generality that each coefficient matrix $D^{(l)}$ is in upper 
row echelon form. This operation corresponds to taking appropriate linear combinations of 
the congruences comprising \eqref{5.5}. Here, we stress that the property that each 
$D^{(l)}$ is highly non-singular implies that the first $r_l \times r_l$ submatrix of 
$D^{(l)}$ is now upper triangular. We denote this matrix by $D_0^{(l)}$. Note that the 
power of $p$ dividing the diagonal entries of $D_0^{(l)}$ depends only on the first 
$r_l\times r_l$ submatrices of the original coefficient matrices $D^{(l)}$. In particular, by 
defining $\Ome(D)$ to be the set of all primes dividing any of the determinants of the 
latter submatrices, we ensure that when $p \not\in \Ome(D)$, then none of the diagonal 
entries of $D_0^{(l)}$ is divisible by $p$.\par  
	
	We now employ an inductive argument in order to successively reduce the degrees of 
the exponential sums occurring within the mean value
	\[
		B_1(p^h)=p^{-2Kh}\sum_{1\le \ba\le p^h} \prod_{j=1}^w 
|S_k(p^h,\bDel_j)|^{2\tau_j}.
	\] 

	Observe that, as a result of our preparatory manipulations, the 
$r_{\ell} \times r_{\ell}$ coefficient matrics $D^{(l)}$ with $2 \le l \le \ell_w$ are upper 
triangular. Thus, the only exponential sum within the above formula for $B_1(p^h)$ that 
depends on $\ba_w^{[\ell_w]}$ is the one involving $\bDel_w$. In order to save clutter, 
we temporarily drop the modulus $p^h$ in our exponential sums $S_{k}(p^h, \bDel_j)$. 
We may thus write 
	\[
		B_1(p^h) = p^{-2Kh} 
\sum_{\substack{1 \le \ba_j^{[\ell_j]}\le p^h \\ (1 \le j \le w-1)}} 
\left(\prod_{j=1}^{w-1} |S_k(\bDel_j)|^{2\tau_j}\right) 
\sum_{1 \le \ba_w^{[\ell_w]} \le p^h} |S_k(\bDel_w)|^{2\tau_w}.  
	\]
	The inner sum is of the shape considered in Lemma~\ref{L5.1} with $l=\ell_w$. On 
writing $\bd_j = (d_{j,j}^{(l)})_{2 \le l \le \ell_j}$ ($1 \le j \le w$), we thus obtain the 
bound  
	\[
		B_1(p^h) \le p^{-2Kh} \Gam_{p^h}(\bd_w) 
\sum_{\substack{1 \le \ba_j^{[\ell_j]}\le p^h \\ (1 \le j \le w-1)}} 
\left(\prod_{j=1}^{w-1} |S_k(\bDel_j)|^{2\tau_j} \right)
\sum_{1 \le \ba_w^{[\ell_w]} \le p^h} |S_{\ell_w}(\ba_w^{[\ell_w]})|^{2\tau_w}.  
	\]	

	\par	Now suppose that for some index $j$ with $1 \le j \le w-1$ we have the bound 
	\begin{align}\label{5.8}
		B_1(p^h) &\le p^{-2Kh} \Ups_j\prod_{i=j+1}^{w} \Gam_{p^h}(\bd_i) 
\sum_{1 \le \ba_i^{[\ell_i]} \le p^h} |S_{\ell_i}(\ba_i^{[\ell_i]})|^{2\tau_i},  
	\end{align}	
	where
	\begin{align*}
		\Ups_j=\sum_{\substack{1 \le \ba_i^{[\ell_i]}\le p^h \\ (1 \le i \le j)}} 
\left(\prod_{i=1}^{j} |S_k(\bDel_i)|^{2\tau_i} \right) .
	\end{align*}
	Again, since we may assume all coefficient matrices $D^{(l)}$ to be in upper row 
echelon form, the only exponential sum within the mean value defining $\Ups_j$ that 
depends on the vector $\ba_j^{[\ell_j]}$ is the one involving $\bDel_j$. Thus, as in the 
case $j=w$ considered above, we may isolate the exponential sum indexed by $j$ and 
apply Lemma~\ref{L5.1}. As a result, we find that
	\begin{align*}
		\Ups_j&= \sum_{\substack{1 \le \ba_i^{[\ell_i]}\le p^h \\ (1 \le i \le j-1)}} 
\left(\prod_{i=1}^{j-1}|S_k(\bDel_i)|^{2\tau_i}\right)\sum_{1\le \ba_j^{[\ell_j]}\le p^h} 
|S_k(\bDel_j)|^{2\tau_j}\\
		& \le \Gam_{p^h}(\bd_j) \sum_{1 \le \ba_j^{[\ell_j]} \le p^h} 
|S_{\ell_j}(\ba_j^{[\ell_j]})|^{2\tau_j} \sum_{\substack{1 \le \ba_i^{[\ell_i]}\le p^h \\ 
(1 \le i \le j-1)}} \prod_{i=1}^{j-1} |S_k(\bDel_i)|^{2\tau_i}  .
	\end{align*}
	Inserting this bound into \eqref{5.8} reproduces \eqref{5.8} with $j$ replaced by 
$j-1$. We may clearly iterate, and after $w$ steps we find that
	\begin{align*}
		B_1(p^h) &\le p^{-2Kh}  \prod_{j=1}^{w} \Gam_{p^h}(\bd_j) 
\sum_{1 \le\ba_j^{[\ell_j]} \le p^h } |S_{\ell_j}(\ba_j^{[\ell_j]})|^{2\tau_j}. 
	\end{align*}		
	Clearly, the vectors $\ba_j^{[\ell_j]}$ with $1 \le j \le w$ together list the coordinates 
of $\ba$. Since $B_1(q)$ is multiplicative, the assertion of the lemma is now confirmed 
upon taking $\calR(q)$ to be the multiplicative function defined via the formula
	\begin{align*}
		\calR(p^h) = \prod_{j=1}^{w} \Gam_{p^h}(\bd_j).
	\end{align*}	
	In particular, we note that $\calR(p^h)$ depends at most on the coefficient matrices 
$D^{(l)}$, and one has $\calR(p^h)=1$ whenever $p \not\in \Ome(D)$.
\end{proof}

\section{Conclusion of the major arcs analysis}\label{S6}
With Lemma~\ref{L5.2} we are now equipped to engage with our goal of showing that the singular series $\frS = \underset{Q \to \infty}{\lim}\frS(Q)$ converges absolutely. In this context, for each prime number $p$ we define the $p$-adic factor
\begin{equation}\label{6.1}
	\chi_p=\sum_{h=0}^\infty A(p^h).
\end{equation}

\begin{lemma}\label{L6.1}
    Suppose that the coefficient matrices $C^{(l)}$ associated with the system \eqref{1.10} are highly non-singular, and that $r_l \ge r_{l+1}$ for $2 \le l \le k-1$. Also, assume that $s \ge 2K+1$. Then the $p$-adic densities $\chi_p$ exist, the singular series $\frS$ is absolutely convergent, and $\frS=\prod_p\chi_p$. In particular, one has $\frS(Q)=\frS+O(Q^{-\delta})$ for some $\delta>0$. Moreover, if the system \eqref{1.10} has a non-singular $p$-adic solution for all primes $p$, then $\frS\gg 1$.
\end{lemma}

\begin{proof}
     On recalling \eqref{4.7} and \eqref{5.1}, we see that $\frS(Q)=\sum_{1\le q\le Q}A(q)$, and so the estimation of the quantity $A(q)$ is our central focus. The multiplicativity of $A(q)$ allows us to restrict our attention to the cases where $q$ is a prime power. Set $\chi_p(H) = \sum_{h=0}^H A(p^h)$ and $L_p(Q) = \lfloor \log Q / \log p \rfloor$. If the product 
    \begin{align*}
    	\prod_{p \le Q} \chi_p(L_p(Q))
    \end{align*}
    converges absolutely as $Q \rightarrow \infty$, then so does $\frS(Q)$ with the same limit. In such circumstances, one has $\frS=\prod_p\chi_p$. It is therefore sufficient to show that for all primes $p$ the limit
	\begin{align*}
    	\chi_p = \lim_{H \to \infty} \chi_p(H)
	\end{align*}
	exists, and moreover that there exists a positive number $\delta$ having the property that $\chi_p = 1 + O(p^{-1-\delta})$ for all but at most a finite set of primes $p$. 
    
	On recalling \eqref{5.2}, we find from \eqref{4.3} that
	\begin{align*}
    	A(p^h) & \ll (p^h)^{2K-s}\max_{\substack{1 \le \ba \le p^h \\(\ba,p)=1}} \left(\prod_{j=1}^w |S_k(p^h, \bDel_j)|\right)^{(s-2K)/w} A_1(p^h)\\
    	& \ll \max_{\substack{1 \le \ba \le p^h \\(\ba,p)=1}} \left(\prod_{j=1}^w (p^h)^{-1/k +\eps}(p^h, \bDel_j)^{1/k}\right)^{(s-2K)/w} A_1(p^h).
	\end{align*}
	The invertibility of the coordinate transform \eqref{5.4} implies that when 
$(\ba,p^h)=1$, then there is at least one index $j$ with $1\le j\le w$ such that 
$(p^h,\bDel_j)\ll 1$, with an implied constant depending at most on the coefficient 
matrices $C^{(l)}$. Since $s - 2K \ge 1$ and $\eps$ may be taken arbitrarily small, we 
deduce that there is a positive number $c_1$, depending at most on the coefficient 
matrices $C^{(l)}$, having the property that
    \begin{align}\label{6.2}
        A(p^h) \le c_1 p^{-h/(2kw)} A_1(p^h).
    \end{align}

    \par We now wish to apply Lemma~\ref{L5.2}. To this end, we first recall \eqref{5.3} and observe that a summation by parts yields the relation
    \begin{align}\label{6.3}
        \sum_{h=0}^L p^{-{\textstyle \frac{h}{2kw}}} A_1(p^h) =  p^{-{\textstyle \frac{L}{2kw}}} B_1(p^L) + \sum_{h=0}^{L-1} \Big(p^{-{\textstyle \frac{h}{2kw}}}-p^{-{\textstyle \frac{h+1}{2kw}}}\Big)B_1(p^h).
    \end{align}
    Since all coefficients on the right hand side are positive, and also both $B_1(p^h)$ and $B_1^*(p^h)$ are non-negative for all non-negative integers $h$, it follows from Lemma~\ref{L5.2} that we may majorise the right hand side of \eqref{6.3} by replacing $B_1(p^h)$ with $\calR(p^h) B_1^*(p^h)$ for $0\le h\le L$. Set $\calR_p=\max_{h\ge 0} \calR(p^h)$, noting that this maximum exists as $\calR(p^h)$ is an integer which is bounded uniformly for all non-negative integers $h$. Also, in analogy to the definition of $B_1^*(p^h)$, we put 
	\begin{align*}
		A_1^*(p^h) = p^{-2Kh}\sum_{\substack{1\le \ba\le p^h\\(\ba,p)=1}}\prod_{j=1}^w |S_{\ell_j}(p^h, \ba_j^{[\ell_j]})|^{2\tau_j}.
	\end{align*}    
    Thus, another summation by parts shows that the right hand side of \eqref{6.3} is no larger than
    \begin{align*}
		\calR_p \left(p^{-{\textstyle \frac{L}{2kw}}} B^*_1(p^L) + \sum_{h=0}^{L-1} \Big(p^{-{\textstyle \frac{h}{2kw}}}-p^{-{\textstyle \frac{h+1}{2kw}}}\Big)B^*_1(p^h)\right) = \calR_p \sum_{h=0}^L p^{-{\textstyle \frac{h}{2kw}}} A_1^*(p^h).
	\end{align*}    
	We have therefore established the bound 
    \begin{align}\label{6.4}
		\sum_{h=0}^L p^{-{\textstyle \frac{h}{2kw}}} A_1(p^h) \le \calR_p \sum_{h=0}^L p^{-{\textstyle \frac{h}{2kw}}} A_1^*(p^h).
	\end{align}  	
    
    \par Since $2\tau_j = \ell_j(\ell_j+1)-2 \ge \ell_j^2$ for all $j$, we can infer further from \eqref{4.3} that there exists a positive number $c_2$, depending at most on $\eps$, such that
    \begin{align*}
        A_1^*(p^h) &\le c_2 p^{h \eps}\sum_{\substack{1\le \ba\le p^h\\(\ba,p)=1}} 
\prod_{j=1}^w \left(p^{-h/\ell_j}(p^h, \ba_j^{[\ell_j]})^{1/\ell_j}\right)^{2\tau_j}
        \le c_2 p^{h \eps}\sum_{\substack{1\le \ba\le p^h\\(\ba,p)=1}}\prod_{j=1}^w 
p^{-h\ell_j}(p^h, \ba_j^{[\ell_j]})^{\ell_j}.
    \end{align*}
    For a fixed vector $\be \in \Z_{\ge 0}^w$ denote by $\Xi(p^h,\be)$ the number of vectors $\ba\in\Z^r$ satisfying $1\le \ba\le p^h$ and $(p^h,\ba_j^{[\ell_j]})=p^{e_j}$ for $1 \le j \le w$. Then one has
    \[
        A_1^*(p^h)\le c_2p^{h\eps}\sum_{\be}\Xi(p^h,\be)\prod_{j=1}^w p^{(e_j-h)\ell_j} ,
    \]
    where the sum is over all vectors $\be \in \Z^w$ satisfying $0 \le e_j \le h$ and having the property that $e_j =0$ for at least one index $j$. For any fixed $j$, the number of choices for $\ba_j^{[\ell_j]} \in \Z^{\ell_j-1}$ having $1 \le \ba_j^{[\ell_j]} \le p^h$ and $(p^h, \ba_j^{[\ell_j]}) = p^{e_j}$ is at most $p^{(h-e_j)(\ell_j-1)}$. It follows that
    \begin{align*}
        \Xi(p^h, \be) \le \prod_{j=1}^w p^{(h-e_j)(\ell_j-1)},
    \end{align*}
	and hence
	\begin{align}\label{6.5}
		A_1^*(p^h)\le c_2p^{h \eps}\sum_{\substack{0\le \be \le h \\ e_1\cdots e_w=0}} \prod_{j=1}^w p^{e_j-h} \le c_2 wp^{h \eps-h}\Bigg(\sum_{e=0}^h p^{e-h} \Bigg)^{w-1}\le c_2 w2^w (p^h)^{-1+\eps}.
	\end{align}

	\par	On recalling \eqref{6.2}, \eqref{6.4} and 	\eqref{6.5} we find that
    \begin{align*}
       	\sum_{h=1}^{\infty}|A(p^h)|&\le c_1\biggl( -1+\sum_{h=0}^\infty p^{-h/(2kw)}A_1(p^h)\biggr)\\
		&\le c_1\biggl( -1+\calR_p\sum_{h=0}^\infty p^{-h/(2kw)}A_1^*(p^h) \biggr)\\
		&\le c_1(\calR_p-1)+w2^w c_1c_2\calR_p\sum_{h=1}^\infty p^{-h(1+1/(2kw)-\eps)}\ll 1.
    \end{align*}
	It follows that the $p$-adic density $\chi_p$ defined in (\ref{6.1}) exists. In particular, whenever $\calR_p=1$ we have
    \begin{align}\label{6.6}
		\sum_{h=1}^{\infty} |A(p^h)|\le c_3p^{-1-1/(3kw)}
	\end{align}
	for some positive number $c_3$ depending at most on the coefficient matrices $C^{(l)}$. On recalling the conclusion of Lemma~\ref{L5.2}, one sees that $\calR_p=1$ for all primes $p$ with $p\not\in \Ome (D)$, and thus
    \begin{align*}
        \prod_p \left(\sum_{h=0}^\infty |A(p^h)|\right)\ll \prod_p \left(1+p^{-1-1/(3kw)}\right)^{c_3}<\zeta(1+1/(3kw))^{c_3}.
    \end{align*}
	Hence, the singular series $\frS$ converges absolutely and one has $\frS=\prod_p\chi_p$.

	\par
    Furthermore, a standard argument yields
    \begin{align*}
        \chi_p = \lim_{h \to \infty} p^{-h(s-r)} M(p^h),
    \end{align*}
    where $M(q)$ denotes the number of solutions $\bx \in (\Z/q\Z)^s$ of the 
congruences
    \begin{align*}
        c_{j, 1}^{(l)}x_1^l + \ldots + c_{j, s}^{(l)}x_s^l \equiv 0 \pmod{q} \qquad 
(1 \le j \le r_l, \, 2 \le l \le k),
    \end{align*}
    corresponding to the equations \eqref{1.10}. Using again the observation that 
$\calR_p=1$ for all sufficiently large primes $p$, we discern from \eqref{6.6} that there exists an integer $p_0$ with the property that
    \begin{align*}
        1/2 \le \prod_{\substack{p > p_0 \\ p \text{ prime}}} \chi_p \le 3/2.
    \end{align*}
    For the remaining finite set of primes, a standard application of Hensel's lemma shows that $\chi_p>0$ whenever the system \eqref{1.10} possesses a non-singular solution in $\Q_p$. We thus conclude that under the hypotheses of the lemma we have $\frS \gg 1$ as claimed.
\end{proof}

We next demonstrate the existence of the limit
\begin{align*}
    \chi_{\infty} = \lim_{Q \to \infty} \frJ_1(Q).
\end{align*}
With this goal in mind, when $W$ is a positive real number, we introduce the auxiliary mean value
\begin{align*}
	\frJ_1^*(W)=\int_{[-W,W]^r}\prod_{j=1}^s |v_k(\btet_j;1)|\d\bbet .
\end{align*}

\begin{lemma}\label{L6.2}
    Under the hypotheses of Theorem~\ref{T4.1}, there is a positive number $\del$ for which one has $\frJ^*_1(2Q) - \frJ^*_1(Q) \ll Q^{-\del}$, and hence the limit $\chi_{\infty}$ exists. In particular, one has
	\begin{align*}
		\frJ_1(Q)=\chi_\infty+O(Q^{-\del}).
	\end{align*}
	Furthermore, if the system \eqref{1.10} has a non-singular solution inside the real unit cube $(-1,1)^s$, then the singular integral $\chi_{\infty}$ is positive.
\end{lemma}

\begin{proof}
    The first part of the proof is inspired by a singular series argument of Heath-Brown and 
Skorobogatov (see \cite[pages 173 and 174]{HBSk2002}). Recall that 
\[
s_0=t_0u_0+\ldots +t_nu_n,
\]
where the integers $t_j$ are defined by means of \eqref{3.15}. Thus, the hypotheses of 
Theorem~\ref{T4.1} imply that $s \ge 2s_0+1 \ge 2K+1$. Let $\calJ$ denote the set of 
$s_0$-element subsets $\{j_1, \dots, j_{s_0}\}$ of $\{1, \dots, s\}$. When $J \in \calJ$, 
define 
    \begin{align}\label{6.7}
		\scrS_J(Q) &= \sum_{1 \le q \le Q} \msum{q}{\ba} q^{-2s_0}\prod_{j \in J} 
|S_k(q, \bLam_{j})|^2 
	\end{align}
	and
	\[
		\scrJ_J(Q) = \int_{[-Q,Q]^r} \prod_{j \in J} |v_k(\btet_{j}; 1)|^2 \d \bbet.
	\]        
    Set $Y=Q^{6r}$, and define the major arcs $\frM_Y(Q)$ via \eqref{4.2}. By making 
the necessary modifications to our initial analysis of the major arcs, we see from 
\eqref{4.8} that for any $J \in \calJ$ one has
    \begin{align}\label{6.8}
        \int_{\frM_Y(Q)} \prod_{j \in J} |g_k(\bgam_{j}; Y)|^2 \d \balp = Y^{2s_0-K} 
\scrS_J(Q) \scrJ_J(Q) + O(Y^{2s_0-K}Q^{-1}).
    \end{align}
    Note that we have $S_k(1, \boldsymbol{1})=1$ for the term corresponding to $q=1$ 
in \eqref{6.7}. Since all other summands are non-negative, it follows that for any 
$Q\ge 1$ and any $J\in \calJ$, one has
    \begin{align}\label{6.9}
        \scrS_J(Q)  \ge 1.
    \end{align}
	On the other hand, for $Y=Q^{6r}$ the major arcs $\frM_Y(Q)$ are disjoint, and we 
conclude from Theorem~\ref{T3.3} that under the hypotheses of Theorem~\ref{T4.1} we 
have
    \[
        \int_{\frM_Y(Q)} \prod_{j \in J}|g_k(\bgam_{j}; Y)|^2 \d \balp \ll 
I_{s_0,\bk}^{\bv, \bu}(Y) \ll Y^{2s_0-K+\eps}.
    \]
    In combination with \eqref{6.8} and \eqref{6.9} it follows that 
    \begin{align}\label{6.10}
        \max_{J\in \calJ}\scrJ_J(Q) \ll Y^{\eps}.
    \end{align}

	\par Since $Y$ is a power of $Q$, we discern from \eqref{4.4} and \eqref{6.10} via 
\eqref{2.3} that for any $Q>1$ we have
    \begin{align*}
        |\frJ^*_1(2Q) - \frJ^*_1(Q)| 
        &\ll \left(\sup_{|\bbet|>Q} \prod_{j=1}^s |v_k(\btet_j;1)|\right)^{1-2s_0/s}  
\int_{[-2Q,2Q]^r} \prod_{j=1}^s |v_k(\btet_j;1)|^{2s_0/s} \d \bbet\\
        &\ll  Q^{-1/(ks)} \max_{J \in \calJ} \scrJ_J(2Q) \ll Q^{-1/(ks)+\eps}.
    \end{align*}
    Here, we exploited the fact that, since the coefficient matrices $C^{(l)}$ are highly 
non-singular, the condition $|\bbet|>Q$ implies that $|\btet_j| \gg Q$ for some index $j$ 
with $1\le j\le s$. This implies the first statement of the lemma. In particular, the singular 
integral $\chi_\infty$ converges absolutely.\par

    In order to establish the second claim, we follow an argument of Schmidt 
\cite{Sch1982}. When $T \ge 1$, define
    \begin{align*}
        w_T(y) = \begin{cases} T(1-T|y|), & \text{when } |y| \le T^{-1},\\ 0, & 
\text{otherwise}, \end{cases}
    \end{align*}
    and recall that
    \begin{align}\label{6.11}
        w_T(y) = \int_{-\infty}^{\infty}  e(\bet y) 
\left( \frac{\sin(\pi \bet/T)}{\pi \bet/T}\right)^2 \d \bet,
    \end{align}
    where the integral converges absolutely. Set
    \begin{align*}
        \Phi_j^{(l)}(\bx) = c_{j,1}^{(l)} x_1^l + \ldots + c_{j,s}^{(l)} x_s^l \qquad 
(1 \le j \le r_l, \, 2 \le l \le k),
    \end{align*}
    and put
    \begin{align*}
        W_T = \int_{[-1,1]^s} \prod_{l=2}^k \prod_{j=1}^{r_l} w_T(\Phi_j^{(l)}(\bz))
\d \bz.
    \end{align*}
    We adapt the argument of \S11 in Schmidt's work \cite{Sch1982} to show that 
$W_T \to \chi_{\infty}$ as $T \to \infty$.\par

	Set
    \begin{align*}
        \psi_T(\bbet) = \prod_{l=2}^k \prod_{j=1}^{r_l} 
\left( \frac{\sin(\pi \bet_j^{(l)}/T)}{\pi \bet_j^{(l)} /T}\right)^2.
    \end{align*}
    Then in light of \eqref{6.11} a change of the order of integration shows that
    \begin{align*}
        W_T = \int_{\R^r} \left(\prod_{i=1}^s v_k(\btet_i; 1)\right) \psi_T(\bbet) \d \bbet,
    \end{align*}
    and hence
    \begin{align}\label{6.12}
        W_T - \chi_{\infty} = \int_{\R^r} \left(\prod_{i=1}^s v_k(\btet_i; 1)\right) 
(\psi_T(\bbet) - 1) \d \bbet.
    \end{align}
    In order to analyse the integral on the right hand side of \eqref{6.12}, it is convenient 
to consider two domains separately. Write $U_1=[-\sqrt T, \sqrt T]^r$, and set 
$U_2 = \R^r \setminus U_1$.
    From the power series expansion of $\psi_T$ we find that
    \begin{align*}
        0\le 1-\psi_T(\bbet) \ll\min\left\{ 1, \sum_{l=2}^k \sum_{j=1}^{r_l}
(|\bet_j^{(l)}|/T)^2\right\},
    \end{align*}
    whence we discern that the domain  $U_1$ contributes at most
    \begin{align*}
        \sup_{\bbet \in U_1}|1 - \psi_T(\bbet)|\int_{\R^r} \prod_{i=1}^s |v_k(\btet_i; 1) | 
\d \bbet \ll T^{-1}.
    \end{align*}
    Note that in the last step we used our previous insight that the singular integral 
converges absolutely. Meanwhile, the contribution from $U_2$ is bounded above by
    \begin{align*}
        \sum_{i=1}^{\infty} |\frJ^*_1(2^i \sqrt T) - \frJ^*_1(2^{i-1} \sqrt T)| \ll 
\sum_{i=1}^{\infty}(2^i \sqrt T)^{-\del} \ll T^{-\del/2},
    \end{align*}
    for some positive number $\del$  with $\del<1$, where again we took advantage of 
our earlier findings. Thus we infer from \eqref{6.12} that
	\begin{equation}\label{6.13}
		|W_T - \chi_{\infty}| \ll T^{-\del/2}
	\end{equation}
	for all $T \ge 1$, and hence $W_T$ does indeed converge to $\chi_{\infty}$, as 
claimed. 
	
	\par
    Suppose now that the system \eqref{1.10} has a non-singular solution inside 
$(-1,1)^s$. Then it follows from the implicit function theorem that the real manifold 
described by the equations in \eqref{1.10} has positive $(s-r)$-dimensional volume inside 
$(-1,1)^s$. In such circumstances, Lemma~2 of Schmidt \cite{Sch1982} shows that 
$W_T \gg 1$ uniformly in $T$. We therefore deduce from (\ref{6.13}) that $\chi_\infty$ is 
indeed positive, confirming the second claim of the lemma.
\end{proof}

Upon combining \eqref{4.11} with Lemmata \ref{L6.1} and \ref{L6.2}, we 
conclude that
\begin{align*}
	N_{s, \bk}^{\bv, \bu}(X) &=X^{s-K}\left(\frS+O(Q^{-\del})\right) 
\left(\chi_\infty+O(Q^{-\del})\right)+O(X^{s-K-\ome})\\
	&=(\calC+o(1))X^{s-K},
\end{align*}
where $\calC = \chi_{\infty} \prod_p \chi_p$. Moreover, the constant $\calC$ is positive 
whenever the system \eqref{1.10} possesses non-singular solutions in all local fields. This 
confirms the asymptotic formula \eqref{1.11}, and completes our proof of 
Theorem~\ref{T4.1}.

\bibliographystyle{amsbracket}
\providecommand{\bysame}{\leavevmode\hbox to3em{\hrulefill}\thinspace}

\end{document}